\documentclass[11pt]{amsart}						

	\usepackage[pagebackref,  pdfpagemode=FullScreen,  colorlinks=false]{hyperref}

	\usepackage{verbatim}
	\usepackage{mathrsfs, mathtools, amscd, amssymb}		

	\usepackage{tikz}
	\usetikzlibrary{calc}
	\usetikzlibrary{trees}
	\tikzstyle{every picture}=[scale=.35,inner sep=0]

	\usepackage{color}

	\newtheorem{thm}{Theorem}
  	\newtheorem{cor}{Corollary}
  	\newtheorem{lem}{Lemma}
  	\newtheorem{prop}{Proposition}

	\newcommand{\thistheoremname}{}					
	\newtheorem{genericthm}[thm]{\thistheoremname}
	
	\newtheorem*{genericthm*}{\thistheoremname}
	\newenvironment{namedthm*}[1]
	  {\renewcommand{\thistheoremname}{#1}%
	   \begin{genericthm*}}
	  {\end{genericthm*}}	  

	\theoremstyle{definition}

	\theoremstyle{remark}
  	\newtheorem{rem}{Remark}

	\newcommand{\M}{\mathcal{M}}
	\newcommand{\Mbar}{\overline{\mathcal{M}}}
	\newcommand{\leftlongmapsto}{\longleftarrow\!\shortmid}

	\newcounter{case}

	\makeatletter
	\@namedef{subjclassname@2010}{%
	\textup{2010} Mathematics Subject Classification}
	\makeatother

\begin{document}

\title{Classes of Weierstrass points on genus $2$ curves}

\author{Renzo Cavalieri}
\address{Department of Mathematics, Colorado State University \newline \indent Fort Collins, CO 80523, USA}
\email{renzo@math.colostate.edu}

\author{Nicola Tarasca}
\address{Department of Mathematics, Rutgers University \newline \indent Piscataway, NJ 08854, USA}
\email{nicola.tarasca@rutgers.edu}

\subjclass[2010]{ 14H99 (primary),  14C99 (secondary)}
\keywords{Effective cycles on moduli spaces of curves, the strata algebra, hyperelliptic curves.}

\begin{abstract}
We study the codimension $n$  locus of curves of genus $2$ with $n$  distinct marked Weierstrass points  inside the moduli space of genus $2$, $n$-pointed curves, for $n \leq 6$. 
We give a recursive description of 
the classes of the closure of these loci inside the moduli space of stable curves. 
For $n\leq 4$, we express these classes using a generating function over stable graphs indexing the boundary strata of moduli spaces of pointed stable curves.
Similarly, we express the closure of these classes inside the moduli space of curves of compact type for all $n$.
 This is a first step in the study of the structure of hyperelliptic classes in all genera.
\end{abstract}

\maketitle

A classical way to construct subvarieties of moduli spaces of curves is to consider families of curves admitting a map  of a fixed degree $d$ to a rational curve. Loci of all such $d$-gonal curves are often referred to as Hurwitz loci.

Faber and Pandharipande (\cite{MR2120989}) use localization on  moduli spaces of stable maps to exhibit the classes of the closure of Hurwitz loci as elements of the tautological ring. Producing explicit closed formulas is in general an open problem.

The algebra of the boundary strata of moduli spaces of stable pointed curves has been recently  studied and used to produce  relations in the tautological ring and an explicit description  for certain infinite families of classes (\cite{MR3264769}, \cite{JPPZ}, \cite{MOPPZ}, \cite{PPZ-rspin}). In particular, a collection of classes for all genera satisfying the axioms of a semi-simple cohomological field theory can be reconstructed from some initial conditions, and  explicit expressions can be obtained as  sums of decorated stable graphs representing boundary strata classes. 

In this paper we study the loci ${\mathcal{H}yp}_{2,n}$ of genus $2$ curves with $n$ marked Weierstrass points. These loci are irreducible of codimension $n$ inside the moduli space of genus $2$, $n$-pointed curves $\mathcal{M}_{2,n}$, and their closures may be studied as (a multiple of) the image of the natural forgetful morphism from a space of admissible covers. When non-empty, the class $\left[\overline{\mathcal{H}yp}_{2,n}\right]$ spans an extremal ray of the cone of effective classes of codimension $n$ on $\overline{\mathcal{M}}_{2,n}$ (\cite{CT2}). Hence, an explicit expression for $\left[\overline{\mathcal{H}yp}_{2,n}\right]$ in terms of the standard generators of the tautological ring gives us a bound on all effective classes of codimension $n$ on $\overline{\mathcal{M}}_{2,n}$.

More in general, one would like to consider loci of hyperelliptic curves with marked Weierstrass points and hyperelliptic conjugate pairs. 
These classes exhibit a structure which is similar to that of a  semi-simple cohomological field theory; hence one would like to  express hyperelliptic classes via {\it graph formulas}: sums of decorated stable graphs, such that the decorations are assigned in a uniform way for all graphs. 
Such graph formulas are desirable, since they are compact, programmable, and they do not require choosing a basis for the tautological groups.  It is not clear that hyperelliptic classes admit an expression of this type in general. We focus here on the classes of the loci $\overline{\mathcal{H}yp}_{2,n}$, which are the first non-trivial examples of hyperelliptic classes.
We now outline  the main results of this article: a graph formula  for the class of the closure of ${\mathcal{H}yp}_{2,n}$ inside the moduli space of curves of compact type ${\mathcal{M}}_{2,n}^{ct}$ (for all~$n$), and for the class of the closure $\overline{\mathcal{H}yp}_{2,n}$ in $\Mbar_{2,n}$ (for $n\leq 4$). We refer the reader who is not familiar with tautological classes to \S \ref{bano} for the necessary notation.

\subsection{Graph formula in compact type}
To study the class $\left[\overline{\mathcal{H}yp}_{2,n}\right]$, we realize the locus $\overline{\mathcal{H}yp}_{2,n}$ as a component of the intersection of $\pi_n^*\left(\overline{\mathcal{H}yp}_{2,n-1}\right)$ and $\rho_n^*\left(\overline{\mathcal{H}yp}_{2,1}\right)$. 
Here, \linebreak $\pi_n\colon \overline{\mathcal{M}}_{g,n}\rightarrow \overline{\mathcal{M}}_{g,n-1}$ is the map forgetting the $n$-th marked point, and $\rho_n \colon \overline{\mathcal{M}}_{g,n}\rightarrow \overline{\mathcal{M}}_{g,1}$ is the map  forgetting all but the $n$-th marked point.
After analyzing all components of such intersection, we obtain a recursive description of the class $\left[\overline{\mathcal{H}yp}_{2,n}\right]$ (see Proposition \ref{rec}).
This approach has been used in \cite{CT2} in the case $n\leq 3$. 

As we intersect the locus $\pi_n^*\left(\overline{\mathcal{H}yp}_{2,n-1}\right)$ with the divisor $\rho_n^*\left(\overline{\mathcal{H}yp}_{2,1}\right)$, we do not realize an excess intersection along any of the components of the intersection. While this simplifies the study of the intersection, we pay a price by breaking the symmetry among the marked points. As a result, the outcome of the recursive description is an expression for $\left[\overline{\mathcal{H}yp}_{2,n}\right]$ not symmetric in the marked points. 

To obtain a symmetric expression, one can use relations in the tautological ring, and thus add a contribution which is trivial modulo rational equivalence. To symmetrize the class of the closure of $\mathcal{H}yp_{2,n}$ in ${\mathcal{M}}_{2,n}^{ct}$, we use some manipulations on rational trees (\S \ref{sec:comb}) to obtain the following result.

\begin{thm}
\label{mainthmctintro}
For $n\leq 6$, the class of the closure of $\mathcal{H}yp_{2,n}$ in ${\mathcal{M}}_{2,n}^{ct}$ is 
the degree $n$ component of the following class:
\begin{multline*}
\sum_{\Gamma\in {G}_{2,n}^{ct}}  \frac{1}{|{\rm Aut}(\Gamma)|} {\xi_{\Gamma}}_*
\left(\sum_{j=0}^n j! 
\left[
\prod_{i=1}^n \left(1+3\omega_i\right) 
\prod_{v\in V(\Gamma)} e^{-t\lambda} 
\right.\right.\\
\left.\left.
\prod_{e=(h,h')\in E_{1}(\Gamma)} \frac{1-e^{t(\psi_{h}+\psi_{h'})}}{\psi_{h}+\psi_{h'}} 
\prod_{e=(h,h')\in E_2(\Gamma)} \frac{1}{\psi_h-(1+3\omega_{h'})} 
\right]_{t^{j}}
\right).
\end{multline*}
\end{thm}

In this formula, $G_{2,n}^{ct}$ denotes the set of isomorphism classes of trees dual to genus $2$, $n$-pointed curves of compact type. For a graph $\Gamma$, $V(\Gamma)$ is the vertex set of $\Gamma$; $E_1(\Gamma)$ and $E_2(\Gamma)$ partition the set of edges into those that separate two subgraphs of genus $1$, and those that separate a rational tree from a subgraph of genus $2$; $(h,h')$ denotes the pair of half-edges that form a given edge $e$. For $e\in E_2(\Gamma)$, the formula is not symmetric in $(h,h')$: we define $h$ to be the half-edge that points {\it towards} the subgraph of genus $2$,
 and $h'$ the half-edge that points {\it outward}.
We refer the reader to \S \ref{bano} for further details on tautological classes and decorated graphs.

The proof of Theorem \ref{mainthmctintro} remains valid also when $n>6$: the degree $n$ component of the class in the statement  represents the class of an empty locus when $n>6$,  hence vanishes in the tautological ring of ${\mathcal{M}}_{2,n}^{ct}$.

\subsection{Formulas beyond compact type} 
It is more challenging to symmetrize the formula beyond compact type for all $n\leq 6$ by using tautological relations. We achieve this for $n\leq 4$ in \S \ref{sec:NCTn} by using relations in the tautological ring of moduli spaces of rational curves, and an identity for divisor classes on moduli spaces of elliptic curves. The resulting expression is equal to the one in Theorem \ref{mainthmctintro}, provided we extend the sum over a larger set of graphs.

We define $\widetilde{G}_{2,n}$ as the set of isomorphism classes of stable graphs $\Gamma$ with the following properties (see Figure \ref{fig:forbiddengraphs}):
\begin{itemize}

\item[i)] $\Gamma$ has no loops of length one;

\item[ii)] if $\Gamma$ has $3$ non-disconnecting edges, then $\Gamma$ has at least $3$ vertices adjacent to those edges;

\item[iii)] if a vertex $v$ of $\Gamma$ is adjacent to a non-disconnecting edge and is not connected to an elliptic vertex via a path of disconnecting edges, then either $v$ is not trivalent and rational at the same time, or $v$ is adjacent to a vertex not trivalent and rational at the same time;

\item[iv)] if a rational vertex $v$ of $\Gamma$ is adjacent to a non-disconnecting edge, then $v$ is not adjacent to two external, trivalent, rational vertices.

\end{itemize}
In particular, $\widetilde{G}_{2,n}$ contains the set of graphs of compact type ${G}^{ct}_{2,n}$. 

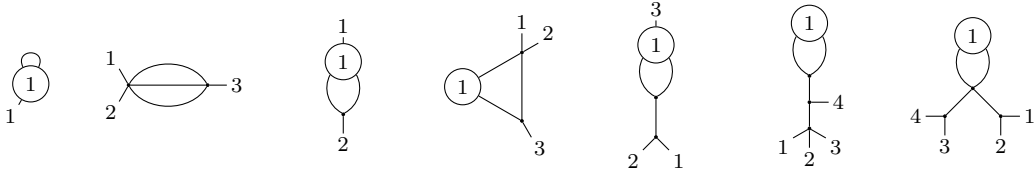
\begin{figure}[htb]
\[
\mbox{
\begin{tikzpicture}[baseline={([yshift=-.5ex]current bounding box.center)}]
      \path(1,0) ellipse (1 and 1);
      \tikzstyle{level 1}=[counterclockwise from=-120,level distance=10mm,sibling angle=60]
      \node [draw,circle,inner sep=2.5] (A0) at (0:1) {$\scriptstyle{1}$} 
      child {node [label=-120: {$\scriptstyle{1}$}]{}};
      \draw (A0) .. controls +(65.000000:1.5) and +(115.000000:1.5) .. (A0);
    \end{tikzpicture}
    }
\quad    
\mbox{    
\begin{tikzpicture}[baseline={([yshift=-.5ex]current bounding box.center)}]
      \path(0,0) ellipse (2 and 1);
      \tikzstyle{level 1}=[counterclockwise from=0,level distance=8mm,sibling angle=120]
      \node[draw,circle,fill] (A0) at (0:1.5) {$\scriptstyle{}$}
      child {node [label=0: {$\scriptstyle{3}$}]{}};
      \tikzstyle{level 1}=[clockwise from=-120,level distance=8mm,sibling angle=120]
      \node[draw,circle,fill] (A1) at (180:1.5) {$\scriptstyle{}$}
      child {node [label=-120: {$\scriptstyle{2}$}]{}}
      child {node [label=120: {$\scriptstyle{1}$}]{}};            
      \path (A0) edge [bend left=-60.000000] node[auto,near start,above=1]{}(A1);
      \path (A0) edge [bend left=60.000000] node[auto,near start,below=1]{}(A1);
      \path (A0) edge [bend left=0.000000] node[auto,near start,below=1]{}(A1);
         \end{tikzpicture}
}        
\quad    
\mbox{    
\begin{tikzpicture}[baseline={([yshift=-.5ex]current bounding box.center)}]
      \path(0,0) ellipse (2 and 1);
      \tikzstyle{level 1}=[counterclockwise from=-90,level distance=8mm,sibling angle=120]
      \node[draw,circle,fill] (A0) at (-90:1) {$\scriptstyle{}$}
      child {node [label=-90: {$\scriptstyle{2}$}]{}};
      \tikzstyle{level 1}=[counterclockwise from=90,level distance=10mm,sibling angle=120]
      \node[draw,circle,inner sep=2.5] (A1) at (90:1) {$\scriptstyle{1}$}
      child {node [label=90: {$\scriptstyle{1}$}]{}};
      \path (A0) edge [bend left=-45.000000] node[auto,near start,above=1]{}(A1);
      \path (A0) edge [bend left=45.000000] node[auto,near start,below=1]{}(A1);
         \end{tikzpicture}
}
\quad    
\mbox{
\begin{tikzpicture}[baseline={([yshift=-.5ex]current bounding box.center)}]
      \path(0,0) ellipse (2 and 2);
      \tikzstyle{level 1}=[counterclockwise from=-60,level distance=10mm,sibling angle=120]
      \node [draw,circle,inner sep=2.5] (A0) at (180:1.5) {$\scriptstyle{1}$};
      \tikzstyle{level 1}=[counterclockwise from=30,level distance=8mm,sibling angle=60]
      \node [draw,circle,fill] (A1) at (60:1.5) {$\scriptstyle{}$}
      child {node [label=30: {$\scriptstyle{2}$}]{}}
      child {node [label=90: {$\scriptstyle{1}$}]{}};
      \tikzstyle{level 1}=[counterclockwise from=-60,level distance=8mm,sibling angle=60]
      \node [draw,circle,fill] (A2) at (-60:1.5) {$\scriptstyle{}$}
      child {node [label=-60: {$\scriptstyle{3}$}]{}};
      \path (A0) edge [] (A1);
      \path (A0) edge [] (A2);
      \path (A1) edge [] (A2);
    \end{tikzpicture}
}
\quad    
\mbox{
\begin{tikzpicture}[baseline={([yshift=-.5ex]current bounding box.center)}]
      \path(0,0) ellipse (2 and 1);
      \tikzstyle{level 1}=[clockwise from=-45,level distance=8mm,sibling angle=90]
      \node[draw,circle,fill] (A0) at (-90:1.5) {$\scriptstyle{}$}
      child {node [label=-45: {$\scriptstyle{1}$}]{}}
      child {node [label=-135: {$\scriptstyle{2}$}]{}};
      \tikzstyle{level 1}=[counterclockwise from=0,level distance=8mm,sibling angle=120]
      \node[draw,circle,fill] (A1) at (0:0) {$\scriptstyle{}$};
      \tikzstyle{level 1}=[counterclockwise from=90,level distance=10mm,sibling angle=120]
      \node[draw,circle,inner sep=2.5] (A2) at (90:2) {$\scriptstyle{1}$}
      child {node [label=90: {$\scriptstyle{3}$}]{}};
      \path (A2) edge [bend left=-45.000000] node[auto,near start,above=1]{}(A1);
      \path (A2) edge [bend left=45.000000] node[auto,near start,below=1]{}(A1);
      \path (A0) edge [] (A1);
         \end{tikzpicture}
}
\quad
\mbox{
\begin{tikzpicture}[baseline={([yshift=-.5ex]current bounding box.center)}]
      \path(0,0) ellipse (2 and 1);
      \tikzstyle{level 1}=[counterclockwise from=-150,level distance=8mm,sibling angle=60]
      \node[draw,circle,fill] (A0) at (-90:2) {$\scriptstyle{}$}
      child {node [label=-150: {$\scriptstyle{1}$}]{}}
      child {node [label=-90: {$\scriptstyle{2}$}]{}}
      child {node [label=-30: {$\scriptstyle{3}$}]{}};
      \tikzstyle{level 1}=[counterclockwise from=0,level distance=8mm,sibling angle=60]
      \node[draw,circle,fill] (A1) at (-90:1) {$\scriptstyle{}$}
      child {node [label=0: {$\scriptstyle{4}$}]{}};
      \tikzstyle{level 1}=[counterclockwise from=0,level distance=8mm,sibling angle=120]
      \node[draw,circle,fill] (A2) at (0:0) {$\scriptstyle{}$};
      \tikzstyle{level 1}=[counterclockwise from=90,level distance=10mm,sibling angle=120]
      \node[draw,circle,inner sep=2.5] (A3) at (90:2) {$\scriptstyle{1}$};
      \path (A3) edge [bend left=-45.000000] node[auto,near start,above=1]{}(A2);
      \path (A3) edge [bend left=45.000000] node[auto,near start,below=1]{}(A2);
      \path (A1) edge [] (A2);
      \path (A0) edge [] (A1);
         \end{tikzpicture}
}
\quad
\mbox{    
\begin{tikzpicture}[baseline={([yshift=-.5ex]current bounding box.center)}]
      \path(0,0) ellipse (2 and 1);
      \tikzstyle{level 1}=[counterclockwise from=-60,level distance=10mm,sibling angle=120]
      \node[draw,circle,inner sep=2.5] (A0) at (90:2) {$\scriptstyle{1}$};
      \node[draw,circle,fill] (A1) at (0:0) {$\scriptstyle{}$};
      \tikzstyle{level 1}=[counterclockwise from=-90,level distance=8mm,sibling angle=90]
      \node[draw,circle,fill] (A2) at (-45:1.5) {$\scriptstyle{}$}
      child {node [label=-90: {$\scriptstyle{2}$}]{}}
      child {node [label=0: {$\scriptstyle{1}$}]{}};      
      \tikzstyle{level 1}=[counterclockwise from=180,level distance=8mm,sibling angle=90]
      \node[draw,circle,fill] (A3) at (-135:1.5) {$\scriptstyle{}$}
      child {node [label=180: {$\scriptstyle{4}$}]{}}
      child {node [label=-90: {$\scriptstyle{3}$}]{}};      
      \path (A0) edge [bend left=-45.000000] node[auto,near start,above=1]{}(A1);
      \path (A0) edge [bend left=45.000000] node[auto,near start,below=1]{}(A1);
      \path (A1) edge [] (A2);
      \path (A1) edge [] (A3);
         \end{tikzpicture}
}
\] 
    \caption{Examples of graphs ruled out by the conditions i--iv defining $\widetilde{G}_{2,n}$. The first two graphs violate respectively i and ii; the third through the sixth graphs are ruled out by iii, and the last graph by iv.}
    \label{fig:forbiddengraphs}
\end{figure}

\begin{thm}
\label{mainthm<=4}
For $n\leq 4$, the class of the closure of $\mathcal{H}yp_{2,n}$ in $\overline{\mathcal{M}}_{2,n}$ is 
the degree $n$ component of the following class:
\begin{multline*}
\sum_{\Gamma\in \widetilde{G}_{2,n}}  \frac{(-1)^{h^1(\Gamma)}}{|{\rm Aut}(\Gamma)|} {\xi_{\Gamma}}_*
\left(\sum_{j=0}^n j! 
\left[
\prod_{i=1}^n \left(1+3\omega_i\right) 
\prod_{v\in V(\Gamma)} e^{-t\lambda} 
\right.\right.\\
\left.\left.
\prod_{e=(h,h')\in E_{1}(\Gamma)} \frac{1-e^{t(\psi_{h}+\psi_{h'})}}{\psi_{h}+\psi_{h'}} 
\prod_{e=(h,h')\in E_2(\Gamma)} \frac{1}{\psi_h-(1+3\omega_{h'})} 
\right]_{t^{j}}
\right).
\end{multline*}
\end{thm}

The graphs of non-compact type in $\widetilde{G}_{2,n}$ and their contributions   are explicitly described in \S \ref{subsec:NCTn}.
We  provide an explicit expression in the case $n=5$ in \S \ref{subsec:NCTn}. At the current stage we have not been able to fully symmetrize this expression using tautological relations. Also for $n=6$, the recursive description in \S \ref{sec:rec} together with Theorem \ref{thm:phigamma} and the analysis in \S \ref{sec:NCTn} produce an explicit expression  for the class $\left[\overline{\mathcal{H}yp}_{2,6}\right]$, although not an immediately symmetric one. 
Symmetrizing these expressions by using tautological  relations  is a possible approach also for $n=5,6$, although a laborious one. We expect that a more conceptual understanding of the structure of these classes  would be a better avenue to obtain symmetric expressions. 
In particular, we pose the question of whether the symmetric expression in Theorem \ref{mainthm<=4} could hold for all $n\leq 6$ for an appropriate  definition of the set $\widetilde{G}_{2,n}$ --- possibly distinct than the one given above when $n=5,6$.

Many of our manipulations on graph formulas for classes on moduli spaces of curves do not require the genus to be $2$ (see for instance \S \ref{sec:prod} and \S \ref{sec:comb}). We envision that our approach could help more generally  find the structure of hyperelliptic classes for arbitrary genera.

\bigskip
\noindent \textit{Acknowledgements.} This project started during the program Combinatorial Algebraic Geometry at the Fields Institute. We would like to thank the organizers and the institute for the excellent working environment. 
We would also like to thank Dawei Chen and Diane Maclagan for helpful discussions, and Nicola Pagani for showing us how to draw stable graphs with the package \texttt{tikz}.

\section{Background and Notation}
\label{bano}
\subsection{Stable graphs and strata classes}
\label{bano1}
We recall  here some of the standard notation for dual graphs of curves; a more comprehensive description can be found, for instance, in \cite{MR3264769}.

Let $G^{rt}_{g,n}$ (respectively $G^{ct}_{g,n}$, or $G_{g,n}$) be the set of isomorphism classes of graphs dual to curves with rational tails (respectively curves of compact type, or stable curves) of genus $g$ with $n$ marked points.

Given a stable graph $\Gamma$, let $V(\Gamma)$ be the set of vertices, endowed with a genus function $g\colon V(\Gamma) \rightarrow \mathbb{Z}_{\geq 0}$ and a valence function $n\colon V(\Gamma) \rightarrow \mathbb{Z}_{\geq 0}$ mapping a vertex $v$ to the valence of $\Gamma$ at $v$; 
let $L(\Gamma)$ be the set of legs, endowed with a bijection to a set of markings; let $H(\Gamma)$ be the set of half-edges, and $F(\Gamma) = L(\Gamma) \cup H(\Gamma)$ be the set of flags of $\Gamma$;
for $g/2\leq i \leq g$, let $E_i(\Gamma)$ denote the set of edges connecting a graph of genus $i$ and a graph of genus $g-i$; the genus of $\Gamma$ is defined as $\sum_{v\in V(\Gamma)} g(v)+h^1(\Gamma).$

A stable graph $\Gamma$ identifies a boundary stratum equal to the image of the  degree $|\textrm{Aut}(\Gamma)|$ glueing map
\[
\xi_{\Gamma} \colon \prod_{v\in V(\Gamma)} \overline{\mathcal{M}}_{g(v), n(v)} =:  \overline{\mathcal{M}}_\Gamma \longrightarrow \overline{\mathcal{M}}_{g,n}.
\]

We also use the following forgetful maps.
Let $\pi_i\colon \overline{\mathcal{M}}_{g,n}\rightarrow  \overline{\mathcal{M}}_{g,[n]\setminus\{i\}}\cong \overline{\mathcal{M}}_{g,n-1}$ be the map forgetting the $i$-th marked point, and $\rho_i \colon \overline{\mathcal{M}}_{g,n}\rightarrow\overline{\mathcal{M}}_{g,\{i\}}\cong \overline{\mathcal{M}}_{g,1}$ be the map  forgetting all but the $i$-th marked point.

\subsection{Divisor classes and decorations}
\label{subsec:dec}
Let $\lambda$ be the first Chern class of the Hodge bundle on $\overline{\mathcal{M}}_{g,n}$, and let $\delta_i$ be the class of the union of the boundary strata of curves with a component of genus $i$ meeting transversally a component of genus $g-i$, for $0\leq i \leq g$. We denote by $\delta$ the class of the boundary of $\overline{\mathcal{M}}_{g,n}$, i.e.~$\delta = \sum_i \delta_i$. The class $\psi_i$ is the first Chern class of the cotangent line at the $i$-th marked point, for $i=1,\dots, n$.
We define $\omega_i:= \rho_i^* (\psi)$, where $\psi$ is the $\psi$-class on $\Mbar_{g,\{i\}}$, for $i=1,\dots, n$. 

Given a stable graph $\Gamma$ and a flag $h\in F(\Gamma)$, let $\psi_h \in A^1(\overline{\mathcal{M}}_{\Gamma})$ be the $\psi$-class at the corresponding marked point, if $h$ is a leg, or at the  shadow 
\footnote{Given a node $n\in C$, we call {\it shadows} of $n$ the two inverse images of $n$ in the normalization map $\nu: \tilde{C}\to C$. Shadows of nodes of $C$ are naturally in bijection with half-edges of the dual graph of $C$.}
of the  node corresponding to $h$, if $h$ is a half-edge. 

The following notation is motivated by the  pull-back of the class $\omega_i$ via  the glueing morphisms $\xi_\Gamma$ (\cite[Lemma 1.9]{blankers2017intersections}).
For a stable graph $\Gamma$, let $C(\Gamma)$ denote the minimum connected genus $g$ subgraph of $\Gamma$ ($C(\Gamma)$ is topologically equivalent to the stable graph obtained after forgetting all the legs of $\Gamma$, and it may have fewer vertices). Denote by $F^{out}(\Gamma)\subset F(\Gamma)$  the subset of flags of $\Gamma \smallsetminus C(\Gamma)$ that {\it point outward}: i.e.~a flag $h$ belongs to $F^{out}(\Gamma)$ if and only if the vertex that it is adjacent to belongs to  a connected subgraph of $\Gamma \smallsetminus \{h\}$ containing $C(\Gamma)$. The function
 \[
w\colon F^{out}(\Gamma)\rightarrow F^{out}(\Gamma)
\]
maps $h\in F^{out}(\Gamma)$ to the flag $w(h)$ in the same connected component of $\Gamma \smallsetminus C(\Gamma)$ and adjacent to $C(\Gamma)$.  All outward pointing flags attached to the same maximal external, rational subgraph have equal image under $w$ (see Figure \ref{fig:opf}).

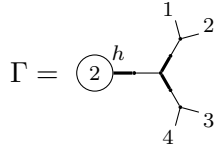
\begin{figure}[htb]
\[\Gamma = 
\mbox{
\begin{tikzpicture}[baseline={([yshift=-.7ex]current bounding box.center)}]
      \path(0,0) ellipse (2 and 2);
      \tikzstyle{level 1}=[counterclockwise from=90,level distance=8mm,sibling angle=60]
      \node [draw,circle,fill] (A0) at (0:0) {};      
      \tikzstyle{level 1}=[counterclockwise from=90,level distance=8mm,sibling angle=60]
      \node [draw,circle,fill] (A01) at (180:1) {};   
      \tikzstyle{level 1}=[counterclockwise from=150,level distance=10mm,sibling angle=60]
          \node [draw,circle,fill] (A12) at (60:0.75) {};  
             \node [draw,circle,fill] (A13) at (-60:0.75) {}; 
      \node [draw,circle,inner sep=2.5] (A1) at (180:2.5) {$\scriptstyle{2}$};
      \tikzstyle{level 1}=[counterclockwise from=15,level distance=8mm,sibling angle=90]
      \node [draw,circle,fill] (A2) at (60:1.5) {}
      child {node [label=15: {$\scriptstyle{2}$}]{}}
      child {node [label=105: {$\scriptstyle{1}$}]{}};
      \tikzstyle{level 1}=[clockwise from=-15,level distance=8mm,sibling angle=90]
      \node [draw,circle,fill] (A3) at (-60:1.5) {}
      child {node [label=-15: {$\scriptstyle{3}$}]{}}
      child {node [label=-105: {$\scriptstyle{4}$}]{}};      
      \path (A0) edge [] (A01);
      \path (A01) edge [very thick] node[auto,above=2, near end, label={90:$\scriptstyle{h}$}]{}(A1);      
            \path (A12) edge [very thick] node[auto,above=2, near end]{}(A0);  
                \path (A13) edge [very thick] node[auto,above=2, near end]{}(A0);  
      \path (A0) edge [] (A2);
      \path (A0) edge [] (A3);             
    \end{tikzpicture}
}
\]
    \caption{The core  $C(\Gamma)$ consists of the vertex of genus $2$. The graph has one maximal external tree with   seven outward pointing flags: the four legs and the the three thickened half-egdes. All outward pointing flags are mapped via $w$ to the half-edge labelled $h$.}
    \label{fig:opf}
\end{figure}

Given $h\in F^{out}(\Gamma)$, let $\omega_h:=\omega_{w(h)} \in A^1(\overline{\mathcal{M}}_{\Gamma})$ be the $\omega$-class corresponding to the flag $w(h)$. In particular, given $i\in L(\Gamma)$ with $i\in w^{-1}(h)$, the class $\omega_h$ is the pull-back of the class $\omega_i\in A^1(\Mbar_{g,n})$ via $\xi_\Gamma$.

\subsection{Decorated graphs} 
\label{nota}
Throughout, we write a stable graph $\Gamma$ decorated with a monomial $\prod_{h\in F(\Gamma)}\psi_h^{k_h}$ as a shorthand for the push-forward ${\xi_\Gamma}_*\left(\prod_{h\in F(\Gamma)}\psi_h^{k_h}\right)$.
For convenience, rational vertices are contracted to points.

We use the following convention on labelling graphs: a graph with some markings omitted stands for the sum of the non-isomorphic graphs obtained by assigning the remaining markings in all possible ways, as illustrated in Figure \ref{fig:dgd}.

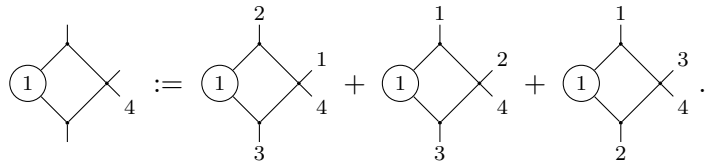
\begin{figure}[htb]
\[
\mbox{
\begin{tikzpicture}[baseline={([yshift=-.7ex]current bounding box.center)}]
      \path(0,0) ellipse (2 and 2);
      \tikzstyle{level 1}=[counterclockwise from=150,level distance=10mm,sibling angle=60]
      \node [draw,circle,inner sep=2.5] (A0) at (180:1.5) {$\scriptstyle{1}$};
      \tikzstyle{level 1}=[counterclockwise from=90,level distance=8mm,sibling angle=60]
      \node [draw,circle,fill] (A1) at (90:1.5) {}
      child {node [label=90: {$\scriptstyle{}$}]{}};
      \tikzstyle{level 1}=[counterclockwise from=-45,level distance=8mm,sibling angle=90]
      \node [draw,circle,fill] (A2) at (0:1.5) {}
      child {node [label=-45: {$\scriptstyle{4}$}]{}}
      child {node [label=45: {$\scriptstyle{}$}]{}};
      \tikzstyle{level 1}=[clockwise from=-90,level distance=8mm,sibling angle=60]
      \node [draw,circle,fill] (A3) at (-90:1.5) {}
      child {node [label=-90: {$\scriptstyle{}$}]{}};      
      \path (A0) edge [] (A1);
      \path (A1) edge [] (A2);
      \path (A2) edge [] (A3);      
      \path (A3) edge [] (A0);            
    \end{tikzpicture}
}
:=
\mbox{
\begin{tikzpicture}[baseline={([yshift=-.7ex]current bounding box.center)}]
      \path(0,0) ellipse (2 and 2);
      \tikzstyle{level 1}=[counterclockwise from=150,level distance=10mm,sibling angle=60]
      \node [draw,circle,inner sep=2.5] (A0) at (180:1.5) {$\scriptstyle{1}$};
      \tikzstyle{level 1}=[counterclockwise from=90,level distance=8mm,sibling angle=60]
      \node [draw,circle,fill] (A1) at (90:1.5) {}
      child {node [label=90: {$\scriptstyle{2}$}]{}};
      \tikzstyle{level 1}=[counterclockwise from=-45,level distance=8mm,sibling angle=90]
      \node [draw,circle,fill] (A2) at (0:1.5) {}
      child {node [label=-45: {$\scriptstyle{4}$}]{}}
      child {node [label=45: {$\scriptstyle{1}$}]{}};
      \tikzstyle{level 1}=[clockwise from=-90,level distance=8mm,sibling angle=60]
      \node [draw,circle,fill] (A3) at (-90:1.5) {}
      child {node [label=-90: {$\scriptstyle{3}$}]{}};      
      \path (A0) edge [] (A1);
      \path (A1) edge [] (A2);
      \path (A2) edge [] (A3);      
      \path (A3) edge [] (A0);            
    \end{tikzpicture}
}
+
\mbox{
\begin{tikzpicture}[baseline={([yshift=-.7ex]current bounding box.center)}]
      \path(0,0) ellipse (2 and 2);
      \tikzstyle{level 1}=[counterclockwise from=150,level distance=10mm,sibling angle=60]
      \node [draw,circle,inner sep=2.5] (A0) at (180:1.5) {$\scriptstyle{1}$};
      \tikzstyle{level 1}=[counterclockwise from=90,level distance=8mm,sibling angle=60]
      \node [draw,circle,fill] (A1) at (90:1.5) {}
      child {node [label=90: {$\scriptstyle{1}$}]{}};
      \tikzstyle{level 1}=[counterclockwise from=-45,level distance=8mm,sibling angle=90]
      \node [draw,circle,fill] (A2) at (0:1.5) {}
      child {node [label=-45: {$\scriptstyle{4}$}]{}}
      child {node [label=45: {$\scriptstyle{2}$}]{}};
      \tikzstyle{level 1}=[clockwise from=-90,level distance=8mm,sibling angle=60]
      \node [draw,circle,fill] (A3) at (-90:1.5) {}
      child {node [label=-90: {$\scriptstyle{3}$}]{}};      
      \path (A0) edge [] (A1);
      \path (A1) edge [] (A2);
      \path (A2) edge [] (A3);      
      \path (A3) edge [] (A0);            
    \end{tikzpicture}
}
+
\mbox{
\begin{tikzpicture}[baseline={([yshift=-.7ex]current bounding box.center)}]
      \path(0,0) ellipse (2 and 2);
      \tikzstyle{level 1}=[counterclockwise from=150,level distance=10mm,sibling angle=60]
      \node [draw,circle,inner sep=2.5] (A0) at (180:1.5) {$\scriptstyle{1}$};
      \tikzstyle{level 1}=[counterclockwise from=90,level distance=8mm,sibling angle=60]
      \node [draw,circle,fill] (A1) at (90:1.5) {}
      child {node [label=90: {$\scriptstyle{1}$}]{}};
      \tikzstyle{level 1}=[counterclockwise from=-45,level distance=8mm,sibling angle=90]
      \node [draw,circle,fill] (A2) at (0:1.5) {}
      child {node [label=-45: {$\scriptstyle{4}$}]{}}
      child {node [label=45: {$\scriptstyle{3}$}]{}};
      \tikzstyle{level 1}=[clockwise from=-90,level distance=8mm,sibling angle=60]
      \node [draw,circle,fill] (A3) at (-90:1.5) {}
      child {node [label=-90: {$\scriptstyle{2}$}]{}};      
      \path (A0) edge [] (A1);
      \path (A1) edge [] (A2);
      \path (A2) edge [] (A3);      
      \path (A3) edge [] (A0);            
    \end{tikzpicture}
}.
\]
    \caption{The convention for interpreting partially  labeled dual graphs. }
    \label{fig:dgd}
\end{figure}

\section{A recursive formula}
\label{sec:rec}

In this section we deduce a recursion that determines the class of $\overline{\mathcal{H}yp}_{2,n}$ by studying the intersection of
$\pi_n^*\left(\overline{\mathcal{H}yp}_{2,n-1}\right)$ and $\rho_n^*\left(\overline{\mathcal{H}yp}_{2,1}\right)$.
In order to describe the components of the intersection,
we introduce the following notation for some effective classes of codimension $n$ in $A^{n}(\overline{\mathcal{M}}_{2,n})$:

\begin{description}
\item[$\Phi_i$ ] For $1\leq i \leq n-1$, we denote by $\Phi_i$ the class of the closure of the locus of curves with an elliptic component meeting in two points $x$ and $y$ a rational component such that  the marked points $p_i$ and $p_n$ are on the rational component and the remaining marked points $p_k$ are on the elliptic component;  we further require that  the condition $2p_k\sim x+y$ is satisfied for all $k\in\{i,n\}^c$.

\item[$\Gamma_{i,j}$] For $1\leq i<j\leq n-1$, we denote by $\Gamma_{i,j}$ the class of the closure of the locus of curves with an elliptic component meeting in two points $x$ and $y$ a rational component such that the marked points $p_i$, $p_j$ and $p_n$ are on the rational component and  the remaining marked points $p_k$ are on the elliptic component; we also require the following divisorial conditions to be satisfied:
\begin{itemize}
\item $2p_i$, $2p_j$, and $x+y$ are in the same pencil of degree two;
\item  $2p_k\sim x+y$ for all $k\in\{i,j,n\}^c$.
\end{itemize}
\end{description}

Recall the maps $\pi_i \colon \overline{\mathcal{M}}_{g,n} \rightarrow \overline{\mathcal{M}}_{g,n-1}$ and $\rho_i \colon \overline{\mathcal{M}}_{g,n} \rightarrow \overline{\mathcal{M}}_{g,1}$, for $i=1,\dots,n$, defined in \S \ref{bano1}. Let $\sigma_i\colon \overline{\mathcal{M}}_{g,n-1} \rightarrow \overline{\mathcal{M}}_{g,n}$ be the $i$-th section, for $i=1,\dots,n-1$. The image of $\sigma_i$ is the boundary divisor of curves with a rational tail containing solely the marked points $p_i$ and $p_n$.

\begin{prop}
\label{rec}
For $2\leq n \leq 6$, one has the following equality in $A^{n}(\overline{\mathcal{M}}_{2,n})$ 
\[
\pi_n^*\left(\overline{\mathcal{H}yp}_{2,n-1}\right)  \cdot  \rho_n^*\left(\overline{\mathcal{H}yp}_{2,1}\right) 
\equiv 
\overline{\mathcal{H}yp}_{2,n} 
+ \sum_{i=1}^{n-1} {\sigma_{i}}_\ast\left( \overline{\mathcal{H}yp}_{2,n-1} \right)
+ \sum_{i=1}^{n-1} \Phi_i
+ \sum_{1\leq i<j\leq n-1} \Gamma_{i,j}.
\]
\end{prop}

\begin{proof}
The cases $n=2,3$ were established in \cite{CT2}. 
We prove the statement for $4\leq n \leq 6$ by induction. 
The intersection of $\pi_n^*\left(\overline{\mathcal{H}yp}_{2,n-1}\right)$ and $\rho_n^*\left(\overline{\mathcal{H}yp}_{2,1}\right)$ consists of curves $(C,p_1,\dots,p_n)$ stably equivalent to a curve with an admissible double cover ramified at $p_1,\dots, p_{n-1}$ and with a possibly different double cover ramified at $p_n$. By definition, the component $\overline{\mathcal{H}yp}_{2,n}$ consists of curves which, up to stable equivalence, have a {\it single} admissible double cover ramified at all marked points. The component $ {\sigma_{i}}_\ast\left( \overline{\mathcal{H}yp}_{2,n-1} \right)$ corresponds to the case when the point $p_n$ coincides with the point $p_i$, for $1\leq i \leq n-1$. The classes $\Phi_i$ (resp.~$\Gamma_{i,j}$) are supported on curves stably equivalent to curves having two {\it distinct} admissible double covers: one ramified at the first $n-1$ marked points, and one ramified at all marks except the point $p_i$ (resp.~except the points $p_i,p_j$). Finally, one checks that there are no other components of codimension $n$ in this intersection. Note that the left-hand side of the formula in the statement is symmetric with respect to the first $n-1$ marked points. It follows that we have
\begin{align}
\label{poush}
\pi_n^*\left(\overline{\mathcal{H}yp}_{2,n-1}\right)  \cdot  \rho_n^*\left(\overline{\mathcal{H}yp}_{2,1}\right) 
\equiv 
a \,\overline{\mathcal{H}yp}_{2,n} 
+ b \sum_{i=1}^{n-1}  {\sigma_{i}}_\ast\left( \overline{\mathcal{H}yp}_{2,n-1} \right)
+ c \sum_{i=1}^{n-1} \Phi_i
+ d \sum_{1\leq i<j\leq n-1} \Gamma_{i,j}
\end{align}
for some coefficients $a,b,c,d$. Forgetting the $(n-1)$-th marked point, and then relabeling the mark $p_n$ as $p_{n-1}$, we obtain\footnote{
The relabeling of the point $p_n$ prevents notation to become overly cumbersome, although it introduces a minor abuse of notation: the symbol $\pi_{n-1}$
denotes two different maps on the opposite sides of (\ref{pusheachterm}). On the left-hand side $\pi_{n-1}\colon\overline{\mathcal{M}}_{2,n} \to \overline{\mathcal{M}}_{2,n-1}$, whereas on the right-hand side $\pi_{n-1}\colon\overline{\mathcal{M}}_{2,n-1}\to \overline{\mathcal{M}}_{2,n-2}$.}

\begin{eqnarray} \label{pusheachterm}
\begin{aligned}
{\pi_{n-1 }}_\ast \left( \pi_n^*\left(\overline{\mathcal{H}yp}_{2,n-1}\right)  \cdot  \rho_n^*\left(\overline{\mathcal{H}yp}_{2,1}\right)  \right) &=&  (6-(n-2))\, \pi_{n-1}^*\left(\overline{\mathcal{H}yp}_{2,n-2}\right)  \cdot  \rho_{n-1}^*\left(\overline{\mathcal{H}yp}_{2,1}\right),  \\
{\pi_{n-1 }}_\ast \left(\overline{\mathcal{H}yp}_{2,n}\right) &=& (6-(n-1))\, \overline{\mathcal{H}yp}_{2,n-1}, \\
{\pi_{n-1 }}_\ast \left( {\sigma_{i}}_\ast \left( \overline{\mathcal{H}yp}_{2,n-1} \right)\right) &=& 
(6-(n-2))\, {\sigma_{i}}_\ast \left( \overline{\mathcal{H}yp}_{2,n-2} \right) \quad\mbox{ for $i<n-1$,}\\
{\pi_{n-1 }}_\ast \left({\sigma_{n-1}}_\ast \left( \overline{\mathcal{H}yp}_{2,n-1} \right) \right) &=& \overline{\mathcal{H}yp}_{2,n-1}, \\
{\pi_{n-1 }}_\ast \left( \Phi_i \right) &=& (4-(n-2)+1)\,\Phi_{i} \quad\mbox{ for $i<n-1$,}\\
{\pi_{n-1 }}_\ast \left( \Phi_{n-1} \right) &=& 0,\\
{\pi_{n-1 }}_\ast \left( \Gamma_{i, j} \right) &=& (4-(n-3)+1)\, \Gamma_{i,j} \quad\mbox{ for $j<n-1$,}\\
{\pi_{n-1 }}_\ast \left( \Gamma_{i, n-1} \right) &=& \Phi_{i}. 
\end{aligned}
\end{eqnarray}
Applying ${\pi_{n-1 }}_\ast$ to \eqref{poush}, one obtains:
\begin{multline}
\label{fourcla}
 (8-n)\, \pi_{n-1}^*\left(\overline{\mathcal{H}yp}_{2,n-2}\right)  \cdot  \rho_{n-1}^*\left(\overline{\mathcal{H}yp}_{2,1}\right)
 \equiv
 \left(a (7-n) +b\right) \overline{\mathcal{H}yp}_{2,n-1} \\
 + b(8-n) \sum_{i=1}^{n-2}  {\sigma_{i}}_\ast \left( \overline{\mathcal{H}yp}_{2,n-2} \right)  
+ \left(c(7-n)+d \right)  \sum_{i=1}^{n-2} \Phi_{i}+
  d(8-n)\sum_{1\leq i<j\leq n-2} \Gamma_{i,j}.
\end{multline}
The inductive hypothesis gives another expression 
 for the  left-hand side of  \eqref{fourcla}
 as a linear combination of the  four classes on the right-hand side  \eqref{fourcla}, with all coefficients equal $(8-n)$.
 One may show that these  classes are numerically independent, for example by exhibiting four classes of complementary degree that produce an upper diagonal matrix of intersection numbers. 
Setting the coefficients of the two linear combinations to agree, we deduce that $a=b=c=d=1$, and the statement follows.
\end{proof}

An expression for the sum of the classes $\Phi_i$ and $\Gamma_{i,j}$ is obtained in Theorem \ref{thm:phigamma}.

\section{Products of divisor classes}
\label{sec:prod}

In this section, we study explicit formulas for certain products of divisor classes on moduli spaces of curves.
Consider a divisor class in ${\rm Pic}(\overline{\mathcal{M}}_{g,n})$  of the form
\[
D = \sum_{i=1}^n a_i \psi_i + c \lambda - b \delta
\]
for some coefficients $a,c,$ and $b$. 
The following formula expresses powers of $D$ as a sum over graphs. The formula appeared in the study of double ramification classes in \cite{JPPZ}. 
Pixton's formula for the exponential of divisor classes is
\begin{align}
\label{pixton}
e^D = \sum_{\Gamma \in G_{g,n}} \frac{1}{|{\rm Aut}(\Gamma)|}{\xi_{\Gamma}}_* \left( 
\prod_{i=1}^n e^{a_i \psi_i} 
\prod_{v\in V(\Gamma)} e^{c\lambda}
\prod_{e=(h,h')\in E(\Gamma)} \frac{1-e^{b(\psi_h + \psi_{h'})}}{\psi_h + \psi_{h'}} \right).
\end{align}

In the next proposition, we apply \eqref{pixton} to solve a similar problem. 
Recall the maps $\rho_i \colon \overline{\mathcal{M}}_{g,n} \rightarrow \overline{\mathcal{M}}_{g,1}$, for $i=1,\dots,n$, defined in \S \ref{bano1}.
Denote by $G^{nrt}_{g,n}\subset G^{ct}_{g,n}$ the subset of  isomorphism classes of graphs of compact type
with no  rational tails.

\begin{prop} 
\label{prod}
Consider a divisor class in ${\rm Pic}({\mathcal{M}}^{ct}_{g,1})$ with $g\geq 1$ of the form
\[
F = a \psi + c \lambda -  b \delta
\]
for some coefficients $a,c,$ and $b$. 
The product
\begin{align}
\label{rhos}
\rho_1^* (F)  \cdots  \rho_n^* (F)
\end{align}
in $A^n({\mathcal{M}}^{ct}_{g,n})$ coincides with the degree $n$ component of the following class
\[
\sum_{\Gamma \in G^{nrt}_{g,n}} \frac{1}{|{\rm Aut}(\Gamma)|}{\xi_{\Gamma}}_* \left( 
\sum_{j=0}^n j! \left[
\prod_{i=1}^n (1+a \omega_i) 
\prod_{v\in V(\Gamma)} e^{ct\lambda}
\prod_{(h,h')\in E(\Gamma)} \frac{1-e^{b t(\psi_h + \psi_{h'})}}{\psi_h + \psi_{h'}} 
\right]_{t^j}
\right).
\]
\end{prop}

\begin{proof}
For $i=1,\dots,n$, one has $\rho_i^*(\psi)=\omega_i$, $\rho_i^*(\lambda)=\lambda$ and $\rho_i^*(\delta)=\delta^{nrt}$, where $\delta^{nrt}$ denotes the sum of classes of all boundary divisors of curves whose dual graphs have two vertices both of positive genus. 
Let $B:= c \lambda -  b\delta^{nrt} \in {\rm Pic}({\mathcal{M}}^{ct}_{g,n})$. 
We can rewrite (\ref{rhos}) as
\[
\prod_{i=1}^n (a\omega_i+B) = \sum_{j=0}^n e_{n-j}(a\omega_1,\dots , a \omega_n) \cdot B^j,
\]
where $e_i$ is the elementary symmetric polynomial of degree $i$. 
We compute the powers $B^j$ by some minor adaptations of  formula (\ref{pixton}): first of all, since $B$ contains no rational tails, no term in any power of $B$ needs to be supported on a graph with rational tails. Hence we restrict our summation to the classes of graphs in $G^{nrt}_{g,n}$. Second, since we are interested in just the power $B^j$, we weight the degree $j$ part of $\eqref{pixton}$ by $j!$. The statement of the proposition follows.
\end{proof}

\section{Combinatorial formulas on  moduli spaces of rational curves}
\label{sec:comb}

In this section we collect some formulas on  $\overline{\mathcal{M}}_{0,n+1}$ that will help us organize the rational tails in the computation of the class of genus $2$ curves with marked, distinct Weierstrass points. We separate them in an independent section, as they  follow from elementary combinatorics of trees. Recall that a \emph{stable tree} is a tree that arises as the dual graph of a stable rational pointed curve. Combinatorially, this means that all vertices are at least trivalent.

Denote by $G_{0,n+1}$ the set of stable trees with $n+1$ labeled leaves. For $T \in G_{0,n+1}$, let $V(T)$ denote the set of vertices,  and $E(T)$ the set of edges of $T$.

\begin{lem}
One has
\begin{align}
\label{chi}
\sum_{T \in G_{0,n+1}} (-1)^{|E(T)|} 
=(-1)^{n}(n-1)!.
\end{align}
\end{lem}

\begin{proof}
After verifying  the formula for $n=2$ by direct inspection, assume it holds  for $n$. Consider the forgetful morphism $\pi_{n+1}^G\colon G_{0,n+1} \to G_{0,n}$, which forgets the $(n+1)$-st leaf, and stabilizes the resulting tree,  if necessary. For a tree $T\in G_{0,n}$, any graph in the inverse image 
$({\pi_{n+1}^G})^{-1}(T)$ is obtained by attaching the $(n+1)$-st leaf at a vertex, or at an internal point of an edge or leaf of $T$. In the last two cases, a new compact edge and a new vertex are formed in the process.
As a result, we have the following equality among formal linear combinations of graphs in $G_{0,n}$
\begin{equation}
\sum_{T \in G_{0,n+1}} (-1)^{|E(T)|} \pi_{n+1}^{G}(T) = \sum_{{T'} \in G_{0,n}} (-1)^{|E(T')|} \left(|V({T'})| -|E({T'})| - n \right) {T'}.
\end{equation}
Formula \eqref{chi} follows by observing $|V({T'})| -|E({T'})| = 1$, and evaluating the above equality via the linearly-extended constant function $G_{0,n} \to \{1\}$.
\end{proof}

\begin{rem} Formula \eqref{chi} may be interpreted as the Euler characteristic of $\mathcal{M}_{0,n+1}^{trop}$, the moduli space of tropical $(n+1)$-pointed rational curves (\cite{Sam}).  There is also a striking similarity to the formula for the Euler characteristic of the uncompactified moduli space of pointed rational curves $\chi({\mathcal{M}}_{0,n+2})=(-1)^{n-1}(n-1)!$; we are not aware of any direct geometric relationship between these two quantities.
\end{rem}

We relabel the $(n+1)$-st leaf as $r$, and view $G_{0,n+1}$ as the set of stable, \textit{rooted} trees with the leaf~$r$ identifying the \textit{root vertex}, and $n$ additional leaves.  For $i = 1, \ldots, n-1$, we denote by
${\sigma^{G}_i}\colon G_{0,n} \to G_{0,n+1}$ the function that 
attaches the leaf $n$ to an internal point of the leaf $i$, thus creating a new edge and a new external, trivalent vertex incident to the leaves $i$ and $n$. 
We denote 
\[
G_{0,n+1}^{ne} := G_{0,n+1} \smallsetminus \bigcup_{i=1}^{n-1} \sigma_i^{G}(G_{0,n}),
\]
i.e.~the set of stable rational trees such that the leg $n$ is not incident to a non-root, external trivalent rational vertex.

\begin{lem}
\label{Gne}
For $n\geq 3$, one has
\[
\sum_{T \in G_{0,n+1}^{ne}} (-1)^{|E(T)|} =0.
\]
\end{lem}

\begin{proof}
The set  $G_{0,n+1}^{ne}$ is the complement in $G_{0,n+1}$ of the \textit{disjoint} union $\bigcup_{i=1}^{n-1} \sigma_i^{G}(G_{0,n})$, and the functions $\sigma_i^G$ are injective.  Recalling that $\sigma_i^G(T)$ has one more edge than $T$, we obtain that
\[
\sum_{T \in G_{0,n+1}} (-1)^{|E(T)|} = 
\sum_{T \in G_{0,n+1}^{ne}} (-1)^{|E(T)|} - (n-1) \sum_{T \in G_{0,n}} (-1)^{|E(T)|}.
\]
The statement thus follows from (\ref{chi}).
\end{proof}

Given a tree $T$ in $G_{0,n+1}^{ne}$, every non-root vertex $v$ has a unique half-edge $h(v)$ \textit{directed towards the root}, i.e.~incident to a half-edge of a connected subtree containing the root. For the root vertex, we define $h(v) = r$. We have thus defined a function 
\[
h\colon V(T) \to F(T).
\]

The next lemma is an immediate geometric reformulation of Lemma \ref{Gne}. 
Let $\overline{\mathcal{M}}_{0,n+1}$ be the moduli space of stable $(n+1)$-pointed rational curves $(R,p_1,\dots, p_n, p_r)$.

\begin{lem}
\label{euler}
One has
\begin{align*}
\label{vanishingtopdeg}
\sum_{T \in G_{0,n+1}^{ne}}
\int_{\overline{\mathcal{M}}_{0,n+1}}  
{\xi_{T}}_* 
\left(
\prod_{v\in V(T)} \frac{1}{\psi_{h(v)}-1}
\right) =0.
\end{align*}
\end{lem}

\begin{proof}
From the string equation, one has 
\[
\int_{\overline{\mathcal{M}}_{0,m}} 
 \frac{1}{\psi_i-1} 
 =- \int_{\overline{\mathcal{M}}_{0,m}} \psi_i^{m-3}=-1,
\]
for all classes $\psi_i$ on $\overline{\mathcal{M}}_{0,m}$. This implies
\begin{align*}
\int_{\overline{\mathcal{M}}_{0,n+1}} 
{\xi_{T}}_* 
\left(
\prod_{v\in V(T)} \frac{1}{\psi_{h(v)}-1}
\right) 
={}& {\xi_{T}}_* 
\left(
\prod_{v\in V(T)} \int_{\overline{\mathcal{M}}_{0,{\rm val}(v)}}  \frac{1}{\psi_{h(v)}-1} 
\right)\\
={}& (-1)^{|V(T)|} 
= (-1)^{|E(T)|+1}
\end{align*}
for each $T$.
The statement thus follows from Lemma \ref{Gne}.
\end{proof}

We will need the following  extension of Lemma \ref{euler}. 
Recall the forgetful map $\pi_n\colon \overline{\mathcal{M}}_{0,n+1} \rightarrow \overline{\mathcal{M}}_{0,n}$.

\begin{lem}
\label{pull-back}
In $A^*(\overline{\mathcal{M}}_{0,n+1})$, one has
\begin{equation}
\label{pullb}
\pi_n^*\left( \sum_{T \in G_{0,n}}
{\xi_{T}}_* 
\left(
\prod_{v\in V(T)} \frac{1}{\psi_{h(v)}-1}
\right) \right)
=
\sum_{T \in G_{0,n+1}^{ne}}
{\xi_{T}}_* 
\left(
\prod_{v\in V(T)} \frac{1}{\psi_{h(v)}-1}
\right).
\end{equation}
\end{lem}

\begin{proof}
Since $\pi_n$ has one-dimensional fibers,  the left-hand side of \eqref{pullb} is a sum of positive dimensional classes. On the  right-hand side, the term of maximal degree  vanishes by Lemma \ref{euler}. Therefore \eqref{pullb} holds in degree $n-2$. 

For the remaining terms, we will prove \eqref{pullb} in three steps: the first is to observe that the left-hand side expands to a sum of trees $T \in G_{0,n+1}^{ne}$ decorated with powers of $\psi$ classes  on half-edges of type $h(v)$ for some $v$, i.e.~the type of summands on the right-hand side  of \eqref{pullb}. Second, we  apply Lemma \ref{euler} ``locally" to identify other groups of terms on the right-hand side of \eqref{pullb} with vanishing contributions.  In the third  step, we identify the remaining terms in the two sides of \eqref{pullb}.

\vskip4pt

\noindent {\bf Step 1.} A summand  $\mu$ of 
\[
\sum_{T \in G_{0,n}}
{\xi_{T}}_* 
\left(
\prod_{v\in V(T)} \frac{1}{\psi_{h(v)}-1}
\right)
\]
consists of a stratum identifying a rooted tree  $T$, and supporting a monomial of the form $\prod_{v\in V(T)} \psi_{h(v)}^{k_v}$ (ignoring the sign, which is irrelevant to this part of the argument). The pull-back $\pi_n^\ast (\mu)$ can be expressed as supported on strata corresponding to trees obtained by adding the $n$-th leaf to one of the vertices of $T$. Denote $\hat{T}$ one of these graphs, and let $\hat{v}$ be the vertex of $\hat{T}$ to which the $n$-th leaf is incident. The contribution of the graph $\hat{T}$ in  $\pi_n^\ast (\mu)$ is 
\begin{align}
\label{contriThat}
{\xi_{\hat{T}}}_* \left(\pi_{n}^\ast \left(\psi_{h(\hat{v})}^{k_{\hat{v}}} \right)\prod_{v\not = \hat{v}} \psi_{h(v)}^{k_v} \right), 
\end{align}
where $\pi_n$  denotes the forgetful morphism $\pi_n\colon \overline{\mathcal{M}}_{0, \textrm{val}(\hat{v})}\to \overline{\mathcal{M}}_{0, \textrm{val}(\hat{v})-1}$.
Using the  relation
 \begin{equation}
 \label{psikock}
\pi_{n}^\ast \left(\psi_{h(\hat{v})}^{k_{\hat{v}}} \right) = \psi_{h(\hat{v})}^{k_{\hat{v}}} - \pi_{n}^\ast \left(\psi_{h(\hat{v})}^{k_{\hat{v}}-1} \right) \textrm{Im}\left(\sigma_{h(\hat{v})}\right),
\end{equation}
where $\sigma_{h(\hat{v})} \colon \overline{\mathcal{M}}_{0, \textrm{val}(\hat{v})-1}\to \overline{\mathcal{M}}_{0, \textrm{val}(\hat{v})}$ is the $h(\hat{v})$-th section,
 one can further decompose \eqref{contriThat} as a sum of two terms, one corresponding to the graph $\hat{T}$, and one to a new graph obtained from $\hat{T}$ by attaching the $n$-th leaf at an interior point of the edge containing the half-edge $h(\hat{v})$. In both cases, the resulting graph belongs to $G_{0,n+1}^{ne}$, and it is decorated by powers of $\psi$ classes at half-edges of type $h(v)$.

\vskip4pt

\noindent {\bf Step 2.} A summand $\mu$ on the right-hand side of \eqref{pullb} consists of a graph $T\in  G_{0,n+1}^{ne}$ and a monomial of the form $(-1)^{|V(T)|}\prod_{v\in V(T)} \psi_{h(v)}^{k_v}$.
For any such term $\mu$, denote by $\hat{v}$ the vertex to which the $n$-th leaf is incident, and by $T_0$ the maximal subtree containing $\hat{v}$ such that each vertex of $T_0$ supports a power of a $\psi$ class of top degree. Note in particular that $T_0 = \varnothing$ if and only if $\psi_{h(\hat{v})}$ does not have degree $\textrm{val}(\hat{v})-3$. We now  fix a monomial $\mu'$ with $T_0 \not= \varnothing$, and restrict our attention to all monomials  $\mu$ that agree with $\mu'$ when restricted to the complement of (their corresponding) $T_0$. Applying Lemma \ref{euler} to the sum of these terms, we obtain that the only summands that survive are those where $\hat{v}$ is trivalent and external in $T_0$. See Figure \ref{fig:step2}. 
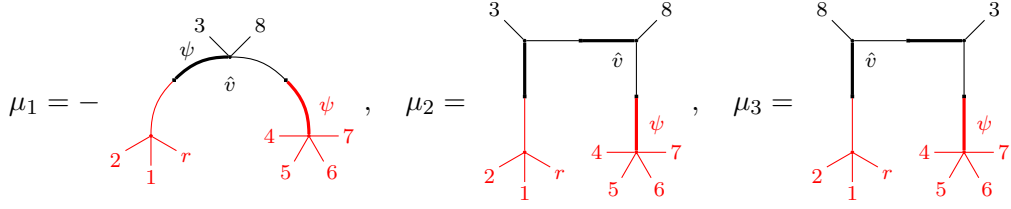
\begin{figure}[tb]
\[
\mu_1= -
\mbox{
\begin{tikzpicture}[baseline={([yshift=-.7ex]current bounding box.center)}]
      \path(0,0) circle (1 and 1);
      \tikzstyle{level 1}=[clockwise from=135,level distance=12mm,sibling angle=90]
      \node[draw,circle, fill] (A0) at (90:3) {} 
      child {node [label=135: {$\scriptstyle{3}$}]{}}
      child {node [label=45: {$\scriptstyle{8}$}]{}};
      \node at ([shift={(-90:1)}]A0) {$\scriptstyle{\hat{v}}$};
      \tikzstyle{level 1}=[counterclockwise from=210,level distance=12mm,sibling angle=60]
      \node[color= red, draw,circle, fill] (A1) at (180:3) {} 
      child [ color= red] {node [label=210: {$\scriptstyle{2}$}]{}}
      child [ color= red] {node [label=-90: {$\scriptstyle{1}$}]{}}      
      child [ color= red] {node [label=-30: {$\scriptstyle{r}$}]{}};
      \tikzstyle{level 1}=[counterclockwise from=180,level distance=12mm,sibling angle=60]
      \node[color= red, draw,circle, fill] (A2) at (0:3) {} 
      child [ color= red] {node [label=180: {$\scriptstyle{4}$}]{}}
      child [ color= red]{node [label=-120: {$\scriptstyle{5}$}]{}}      
      child [ color= red] {node [label=-60: {$\scriptstyle{6}$}]{}}
      child [ color= red] {node [label=0: {$\scriptstyle{7}$}]{}};      
      \node[draw, fill, very thin] (A0A1) at (135:3) {}; 
      \node[draw, fill, very thin] (A0A2) at (45:3) {};       
      \path (A0) edge [very thick,  bend left=-22.500000] node[auto,above=3, label={180:$\scriptstyle{\psi}$}]{}(A0A1);
      \path (A0A1) edge [ color= red, bend left=-22.500000] node[auto,near start,above=3, label={180:}]{}(A1);      
      \path (A0) edge [bend left=22.500000] node[auto,right=3, label={0:}]{}(A0A2);
      \path (A0A2) edge [color= red, very thick, bend left=22.500000] node[auto,right=3, label={0:$\scriptstyle{\psi}$}]{}(A2);      
         \end{tikzpicture}
}
,
\quad \mu_2 =
\mbox{
\begin{tikzpicture}[baseline={([yshift=-.7ex]current bounding box.center)}]
      \path(0,0) circle (1 and 1);
      \tikzstyle{level 1}=[clockwise from=135,level distance=12mm,sibling angle=90]
      \node[draw,circle, fill] (A0) at (135:3) {} 
      child {node [label=135: {$\scriptstyle{3}$}]{}};
      \tikzstyle{level 1}=[clockwise from=45,level distance=12mm,sibling angle=90]
      \node[draw,circle, fill] (A00) at (45:3) {} 
      child {node [label=45: {$\scriptstyle{8}$}]{}};
      \node at ([shift={(-135:1)}]A00) {$\scriptstyle{\hat{v}}$};
      \tikzstyle{level 1}=[counterclockwise from=210,level distance=12mm,sibling angle=60]
      \node[color= red,draw,circle, fill] (A1) at (225:3) {} 
      child [ color= red]{node [label=210: {$\scriptstyle{2}$}]{}}
      child [ color= red]{node [label=-90: {$\scriptstyle{1}$}]{}}      
      child [ color= red]{node [label=-30: {$\scriptstyle{r}$}]{}};
      \tikzstyle{level 1}=[counterclockwise from=180,level distance=12mm,sibling angle=60]
      \node[color= red,draw,circle, fill] (A2) at (-45:3) {} 
      child [ color= red]{node [label=180: {$\scriptstyle{4}$}]{}}
      child [ color= red]{node [label=-120: {$\scriptstyle{5}$}]{}}      
      child [ color= red] {node [label=-60: {$\scriptstyle{6}$}]{}}
      child [ color= red]{node [label=0: {$\scriptstyle{7}$}]{}};      
      \node[draw, fill, very thin] (A2A00) at ($(A2)!0.5!(A00)$) {}; 
      \node[draw, fill, very thin] (A0A00) at ($(A0)!0.5!(A00)$) {}; 
      \node[draw, fill, very thin] (A0A1) at ($(A0)!0.5!(A1)$) {};       
      \path (A0) edge [very thick, bend left=0.000000] node[auto,near start,above=1]{}(A0A1);
      \path (A0A1) edge [color= red,bend left=0.000000] node[auto,near start,above=1]{}(A1);      
      \path (A0) edge [bend left=0.000000] node[auto,near start,above=1]{}(A0A00);      
      \path (A0A00) edge [very thick, bend left=0.000000] node[auto,near start,above=1]{}(A00);            
      \path (A00) edge [bend left=0.000000] node[auto,near end,right=2]{}(A2A00);
      \path (A2A00) edge [color= red,very thick, bend left=0.000000] node[auto,right=2, label={0:$\scriptstyle{\psi}$}]{}(A2);
         \end{tikzpicture}
}
,
\quad \mu_3 =
\mbox{
\begin{tikzpicture}[baseline={([yshift=-.7ex]current bounding box.center)}]
      \path(0,0) circle (1 and 1);
      \tikzstyle{level 1}=[clockwise from=135,level distance=12mm,sibling angle=90]
      \node[draw,circle, fill] (A0) at (135:3) {} 
      child {node [label=135: {$\scriptstyle{8}$}]{}};
      \tikzstyle{level 1}=[clockwise from=45,level distance=12mm,sibling angle=90]
      \node[draw,circle, fill] (A00) at (45:3) {} 
      child {node [label=45: {$\scriptstyle{3}$}]{}};
      \node at ([shift={(-45:1)}]A0) {$\scriptstyle{\hat{v}}$};
      \tikzstyle{level 1}=[counterclockwise from=210,level distance=12mm,sibling angle=60]
      \node[color= red,draw,circle, fill] (A1) at (225:3) {} 
      child [ color= red]{node [label=210: {$\scriptstyle{2}$}]{}}
      child [ color= red]{node [label=-90: {$\scriptstyle{1}$}]{}}      
      child [ color= red]{node [label=-30: {$\scriptstyle{r}$}]{}};
      \tikzstyle{level 1}=[counterclockwise from=180,level distance=12mm,sibling angle=60]
      \node[color= red,draw,circle, fill] (A2) at (-45:3) {} 
      child [ color= red]{node [label=180: {$\scriptstyle{4}$}]{}}
      child [ color= red]{node [label=-120: {$\scriptstyle{5}$}]{}}      
      child [ color= red]{node [label=-60: {$\scriptstyle{6}$}]{}}
      child [ color= red] {node [label=0: {$\scriptstyle{7}$}]{}};      
      \node[draw, fill, very thin] (A2A00) at ($(A2)!0.5!(A00)$) {}; 
      \node[draw, fill, very thin] (A0A00) at ($(A0)!0.5!(A00)$) {}; 
      \node[draw, fill, very thin] (A0A1) at ($(A0)!0.5!(A1)$) {};       
      \path (A0) edge [very thick, bend left=0.000000] node[auto,near start,above=1]{}(A0A1);
      \path (A0A1) edge [color= red, bend left=0.000000] node[auto,near start,above=1]{}(A1);      
      \path (A0) edge [bend left=0.000000] node[auto,near start,above=1]{}(A0A00);      
      \path (A0A00) edge [very thick, bend left=0.000000] node[auto,near start,above=1]{}(A00);            
      \path (A00) edge [bend left=0.000000] node[auto,near end,right=2]{}(A2A00);
      \path (A2A00) edge [color= red, very thick, bend left=0.000000] node[auto,right=2, label={0:$\scriptstyle{\psi}$}]{}(A2);
         \end{tikzpicture}
}
\]
    \caption{An illustration of \textit{Step 2} in the proof of Lemma \ref{pull-back}. Three terms that share the same complement of $T_0$, colored red. In $\mu_2$, the vertex $\hat{v}$ is trivalent and external in $T_0$. As proven in Lemma \ref{euler}, the sum of the other two terms is $\mu_1+\mu_3 = 0$.}
    \label{fig:step2}
\end{figure}

\vskip4pt

\noindent{\bf Step 3.} We now turn our attention to each remaining term in the right-hand side of \eqref{pullb}, and realize it as arising in a unique way in the expansion of the left-hand side described in \textit{Step 1.} There are two separate cases to consider, illustrated in Figure \ref{fig:step3}.
\begin{figure}[tb]
\begin{eqnarray*}
\mbox{
\begin{tikzpicture}[baseline={([yshift=-.7ex]current bounding box.center)}]
      \path(0,0) circle (1 and 1);
      \tikzstyle{level 1}=[clockwise from=135,level distance=12mm,sibling angle=90]
      \node[draw,circle, fill] (A0) at (135:3) {} 
      child {node [label=135: {$\scriptstyle{3}$}]{}};
      \tikzstyle{level 1}=[clockwise from=45,level distance=12mm,sibling angle=90]
      \node[draw,circle, fill] (A00) at (45:3) {} 
      child {node [label=45: {$\scriptstyle{4}$}]{}};
      \tikzstyle{level 1}=[counterclockwise from=210,level distance=12mm,sibling angle=60]
      \node[draw,circle, fill] (A1) at (225:3) {} 
      child {node [label=210: {$\scriptstyle{2}$}]{}}
      child {node [label=-90: {$\scriptstyle{1}$}]{}}      
      child {node [label=-30: {$\scriptstyle{r}$}]{}};
      \tikzstyle{level 1}=[counterclockwise from=210,level distance=12mm,sibling angle=40]
      \node[draw,circle, fill] (A2) at (-45:3) {} 
      child {node [label=-150: {$\scriptstyle{5}$}]{}}      
      child {node [label=-110: {$\scriptstyle{6}$}]{}}
      child {node [label=-70: {$\scriptstyle{7}$}]{}}
      child {node [label=-30: {$\scriptstyle{8}$}]{}};  
      \node at ([shift={(150:1)}]A2) {$\scriptstyle{\hat{v}}$};          
      \node[draw, fill, very thin] (A2A00) at ($(A2)!0.5!(A00)$) {}; 
      \node[draw, fill, very thin] (A0A00) at ($(A0)!0.5!(A00)$) {}; 
      \node[draw, fill, very thin] (A0A1) at ($(A0)!0.5!(A1)$) {};       
      \path (A0) edge [very thick, bend left=0.000000] node[auto,near start,above=1]{}(A0A1);
      \path (A0A1) edge [bend left=0.000000] node[auto,near start,above=1]{}(A1);      
      \path (A0) edge [bend left=0.000000] node[auto,near start,above=1]{}(A0A00);      
      \path (A0A00) edge [very thick, bend left=0.000000] node[auto,near start,above=1]{}(A00);            
      \path (A00) edge [bend left=0.000000] node[auto,near end,right=2]{}(A2A00);
      \path (A2A00) edge [very thick, bend left=0.000000] node[auto,right=2, label={0:$\scriptstyle{\psi}$}]{}(A2);
         \end{tikzpicture}
}
&\quad \leftlongmapsto \quad&
\mbox{
\begin{tikzpicture}[baseline={([yshift=-.7ex]current bounding box.center)}]
      \path(0,0) circle (1 and 1);
      \tikzstyle{level 1}=[clockwise from=135,level distance=12mm,sibling angle=90]
      \node[draw,circle, fill] (A0) at (135:3) {} 
      child {node [label=135: {$\scriptstyle{3}$}]{}};
      \tikzstyle{level 1}=[clockwise from=45,level distance=12mm,sibling angle=90]
      \node[draw,circle, fill] (A00) at (45:3) {} 
      child {node [label=45: {$\scriptstyle{4}$}]{}};
      \tikzstyle{level 1}=[counterclockwise from=210,level distance=12mm,sibling angle=60]
      \node[draw,circle, fill] (A1) at (225:3) {} 
      child {node [label=210: {$\scriptstyle{2}$}]{}}
      child {node [label=-90: {$\scriptstyle{1}$}]{}}      
      child {node [label=-30: {$\scriptstyle{r}$}]{}};
      \tikzstyle{level 1}=[counterclockwise from=210,level distance=12mm,sibling angle=60]
      \node[draw,circle, fill] (A2) at (-45:3) {} 
      child {node [label=-150: {$\scriptstyle{5}$}]{}}      
      child {node [label=-90: {$\scriptstyle{6}$}]{}}
      child {node [label=-30: {$\scriptstyle{7}$}]{}};  
      \node at ([shift={(150:1)}]A2) {$\scriptstyle{\hat{v}}$};          
      \node[draw, fill, very thin] (A2A00) at ($(A2)!0.5!(A00)$) {}; 
      \node[draw, fill, very thin] (A0A00) at ($(A0)!0.5!(A00)$) {}; 
      \node[draw, fill, very thin] (A0A1) at ($(A0)!0.5!(A1)$) {};       
      \path (A0) edge [very thick, bend left=0.000000] node[auto,near start,above=1]{}(A0A1);
      \path (A0A1) edge [bend left=0.000000] node[auto,near start,above=1]{}(A1);      
      \path (A0) edge [bend left=0.000000] node[auto,near start,above=1]{}(A0A00);      
      \path (A0A00) edge [very thick, bend left=0.000000] node[auto,near start,above=1]{}(A00);            
      \path (A00) edge [bend left=0.000000] node[auto,near end,right=2]{}(A2A00);
      \path (A2A00) edge [very thick, bend left=0.000000] node[auto,right=2, label={0:$\scriptstyle{\psi}$}]{}(A2);
         \end{tikzpicture}
}
\\
\mbox{
\begin{tikzpicture}[baseline={([yshift=-.7ex]current bounding box.center)}]
      \path(0,0) circle (1 and 1);
      \tikzstyle{level 1}=[clockwise from=135,level distance=12mm,sibling angle=90]
      \node[draw,circle, fill] (A0) at (135:3) {} 
      child {node [label=135: {$\scriptstyle{3}$}]{}};
      \tikzstyle{level 1}=[clockwise from=45,level distance=12mm,sibling angle=90]
      \node[draw,circle, fill] (A00) at (45:3) {} 
      child {node [label=45: {$\scriptstyle{8}$}]{}};
      \tikzstyle{level 1}=[counterclockwise from=210,level distance=12mm,sibling angle=60]
      \node[draw,circle, fill] (A1) at (225:3) {} 
      child {node [label=210: {$\scriptstyle{2}$}]{}}
      child {node [label=-90: {$\scriptstyle{1}$}]{}}      
      child {node [label=-30: {$\scriptstyle{r}$}]{}};
      \tikzstyle{level 1}=[counterclockwise from=210,level distance=12mm,sibling angle=40]
      \node[draw,circle, fill] (A2) at (-45:3) {} 
      child {node [label=-150: {$\scriptstyle{4}$}]{}}      
      child {node [label=-110: {$\scriptstyle{5}$}]{}}
      child {node [label=-70: {$\scriptstyle{6}$}]{}}
      child {node [label=-30: {$\scriptstyle{7}$}]{}};  
      \node at ([shift={(225:1)}]A00) {$\scriptstyle{\hat{v}}$};       
      \node at ([shift={(150:1)}]A2) {$\scriptstyle{v^\circ}$};          
      \node[draw, fill, very thin] (A2A00) at ($(A2)!0.5!(A00)$) {}; 
      \node[draw, fill, very thin] (A0A00) at ($(A0)!0.5!(A00)$) {}; 
      \node[draw, fill, very thin] (A0A1) at ($(A0)!0.5!(A1)$) {};       
      \path (A0) edge [very thick, bend left=0.000000] node[auto,near start,above=1]{}(A0A1);
      \path (A0A1) edge [bend left=0.000000] node[auto,near start,above=1]{}(A1);      
      \path (A0) edge [bend left=0.000000] node[auto,near start,above=1]{}(A0A00);      
      \path (A0A00) edge [very thick, bend left=0.000000] node[auto,near start,above=1]{}(A00);            
      \path (A00) edge [bend left=0.000000] node[auto,near end,right=2]{}(A2A00);
      \path (A2A00) edge [very thick, bend left=0.000000] node[auto,right=2, label={0:$\scriptstyle{\psi}$}]{}(A2);
         \end{tikzpicture}
}
&\quad \leftlongmapsto \quad&
\mbox{
\begin{tikzpicture}[baseline={([yshift=-.7ex]current bounding box.center)}]
      \path(0,0) circle (1 and 1);
      \tikzstyle{level 1}=[clockwise from=90,level distance=12mm,sibling angle=90]
      \node[draw,circle, fill] (A0) at (90:3) {} 
      child {node [label=90: {$\scriptstyle{3}$}]{}};
      \tikzstyle{level 1}=[counterclockwise from=210,level distance=12mm,sibling angle=60]
      \node[draw,circle, fill] (A1) at (180:3) {} 
      child {node [label=210: {$\scriptstyle{2}$}]{}}
      child {node [label=-90: {$\scriptstyle{1}$}]{}}      
      child {node [label=-30: {$\scriptstyle{r}$}]{}};
      \tikzstyle{level 1}=[counterclockwise from=210,level distance=12mm,sibling angle=40]
      \node[draw,circle, fill] (A2) at (0:3) {} 
      child {node [label=-150: {$\scriptstyle{4}$}]{}}
      child {node [label=-110: {$\scriptstyle{5}$}]{}}      
      child {node [label=-70: {$\scriptstyle{6}$}]{}}
      child {node [label=-30: {$\scriptstyle{7}$}]{}};      
      \node at ([shift={(150:1)}]A2) {$\scriptstyle{v^\circ}$};       
      \node[draw, fill, very thin] (A0A1) at (135:3) {}; 
      \node[draw, fill, very thin] (A0A2) at (45:3) {};       
      \path (A0) edge [very thick, bend left=-22.500000] node[auto,above=3, label={180:}]{}(A0A1);
      \path (A0A1) edge [bend left=-22.500000] node[auto,near start,above=3, label={180:}]{}(A1);      
      \path (A0) edge [bend left=22.500000] node[auto,right=3, label={0:}]{}(A0A2);
      \path (A0A2) edge [very thick, bend left=22.500000] node[auto,right=3, label={0:$\scriptstyle{\psi^2}$}]{}(A2);      
         \end{tikzpicture}
}
\end{eqnarray*}
    \caption{The two cases in \textit{Step 3} of Lemma \ref{pull-back}. The terms on the left arise from the pull-back of the terms on the right.}
    \label{fig:step3}
\end{figure}
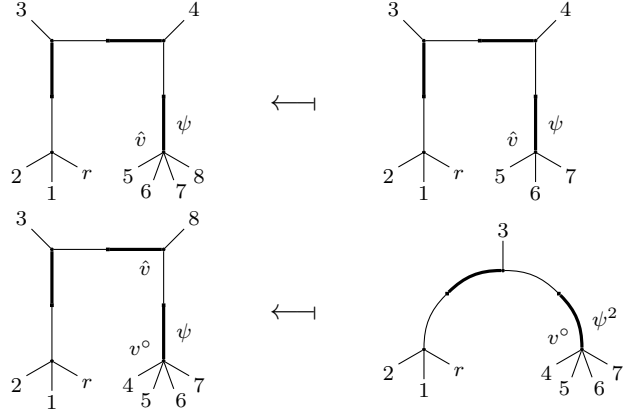

For a term $\mu$ corresponding to a pair $\left(T, (-1)^{|V(T)|}\prod_{v\in V(T)} \psi_{h(v)}^{k_v}\right)$ such that $T_0 = \varnothing$, it must be that $\textrm{val}(\hat{v})\geq 4$ and $k_{\hat{v}}< \textrm{val}({\hat{v}})-3$. Then the term $\mu$ appears as the first term in the  right-hand side of \eqref{psikock} in the pull-back via $\pi_n$ of the term corresponding to $\left(\pi_{n}^G(T), (-1)^{|V(\pi_{n}^G(T))|}\prod_{v\in V(\pi_{n}^G(T))} \psi_{h(v)}^{k_v}\right)$.

Fix a term $\mu$ corresponding to $\left(T, (-1)^{|V(T)|}\prod_{v\in V(T)} \psi_{h(v)}^{k_v}\right)$ where $T_0 \not= \varnothing$. After \textit{Step 2}, $\hat{v}$ is trivalent and external in $T_0$. Denote by $v^{\circ}$ the vertex adjacent to $\hat{v}$ via the edge not containg $h(\hat{v})$ (in other words, the next vertex after $\hat{v}$ moving away from the root).
Since  $v^{\circ}\not\in T_0$, it must be that $k_{v^{\circ}} < \textrm{val}(v^{\circ})-3$. Then the term $\mu$ appears as the second term in the  right-hand side of \eqref{psikock} in the pull-back of the term corresponding to $\left(\pi_{n}^G(T), (-1)^{|V(\pi_{n}^G(T))|} \psi_{h(v^{\circ})}^{k_{v^{\circ}}+1}\prod_{v\not= v^{\circ}} \psi_{h(v)}^{k_v}\right)$.

The statement thus follows by induction on $n$.
\end{proof}

We conclude this section with one more combinatorial identity which will be used in the proof of Theorem \ref{mainthmctintro}. 
Let $B_{0,n+1}\subset G_{0,n+1}$ be the set of trees defined recursively as 
\begin{align*}
B_{0,4} &= \left(\pi^G_{3}\right)^{-1} \left( G_{0,3}\right),\\
B_{0,n+1} &= \left(\pi^G_{n}\right)^{-1} \left(B_{0,n} \cup \bigcup_{i=1}^{n-2} \left(\textrm{Im}\left(\sigma^G_i\right) \cdot B_{0,n} \right)\right), \quad \mbox{for } n\geq 4.
\end{align*}
Here, $\textrm{Im}\left(\sigma^G_i\right) \cdot B_{0,n}$ is the set of stable graphs obtained by letting the $(n-1)$-st leaf collide  with a leaf different than $r$ of a graph in $B_{0,n}$. Note that the $(n-1)$-st leaf of a graph in $B_{0,n}$ is never attached to a trivalent vertex. 

Informally, $B_{0,n+1}$ contains graphs that are obtained from graphs in $B_{0,n}$ by either directly adding the $n$-th leaf to a vertex,  or by first ``pulling out" the $(n-1)$-st leaf and one other non-root leaf adjacent to the same vertex to a trivalent external vertex, and then adding the $n$-th leaf to any vertex of this new graph. 

\begin{prop}
\label{eulergen}
For $n\geq 3$, one has
\begin{align}
\label{eulergeneq}
\sum_{T \in G_{0,n+1}^{ne}}
{\xi_{T}}_* 
\left(
\prod_{v\in V(T)} \frac{1}{\psi_{h(v)}-1}
\right) =
\sum_{T \in B_{0,n+1}}
{\xi_{T}}_* \left((-1)^{|V(T)|}\right).
\end{align}
\end{prop}

\begin{proof}
We prove the statement by recursion. 
The case $n=3$ follows from Lemma \ref{euler}.
For the recursion, we consider the pull-back of (\ref{eulergeneq}) via $\pi_n$.
By the definition of $B_{0,n+1}$, we have
\begin{eqnarray*}
\sum_{T \in B_{0,n+1}}
{\xi_{T}}_*  \left((-1)^{|V(T)|}\right) &=&
\pi_n^*\left( 
\sum_{T \in B_{0,n}}
{\xi_{T}}_*  \left((-1)^{|V(T)|}\right)
\right)\\
&&{}-
\pi_n^*\left( 
\sum_{i=1}^{n-2} \textrm{Im}(\sigma_{i }) \cdot\left(
\sum_{T \in B_{0,n}}
{\xi_{T}}_*  \left((-1)^{|V(T)|}\right)
\right)\right).
\end{eqnarray*}
On the other hand, we have
\begin{align*}
\sum_{T \in G_{0,n+1}^{ne}}
{\xi_{T}}_* 
\left(
\prod_{h\in V(T)} \frac{1}{\psi_{h(v)}-1}
\right) ={}&
\pi_n^*\left( 
\sum_{T \in G_{0,n}}
{\xi_{T}}_* 
\left(
\prod_{h\in V(T)} \frac{1}{\psi_{h(v)}-1}
\right)
\right)\\
={}& \pi_n^*\left( 
\sum_{T \in G_{0,n}^{ne}}
{\xi_{T}}_* 
\left(
\prod_{h\in V(T)} \frac{1}{\psi_{h(v)}-1}
\right)
\right)\\
& {}- \pi_n^*\left( 
\sum_{i=1}^{n-2} \textrm{Im}( \sigma_{i }) \cdot \left(
\sum_{T \in G_{0,n}^{ne}}
{\xi_{T}}_* 
\left(
\prod_{h\in V(T)} \frac{1}{\psi_{h(v)}-1}
\right)
\right)
\right).
\end{align*}
We have used Lemma \ref{pull-back} for the first equality. The second equality follows from the identity
\[
G_{0,n} = G_{0,n}^{ne} \cup \bigcup_{i=1}^{n-2} \left(\textrm{Im}(\sigma^G_i)  \cdot G_{0,n}^{ne} \right),
\]
where $\textrm{Im}(\sigma^G_i)  \cdot G_{0,n}^{ne}=\sigma^G_i  \left(G_{0,n-1}\right)$.
By recursion, we have
\begin{align*}
\sum_{T \in G_{0,n}^{ne}}
{\xi_{T}}_* 
\left(
\prod_{h\in V(T)} \frac{1}{\psi_{h(v)}-1}
\right) =
\sum_{T \in B_{0,n}}
{\xi_{T}}_*  \left((-1)^{|V(T)|}\right).
\end{align*}
The statement thus follows. 
\end{proof}

\section{The formula in compact type}
\label{sec:ct}

In this section, we prove Theorem \ref{mainthmctintro}, that is,
the class of the closure $\mathcal{H}yp^{ct}_{2,n}$ of $\mathcal{H}yp_{2,n}$ in ${\mathcal{M}}_{2,n}^{ct}$ is 
the degree $n$ component of the following class:
\begin{multline}
\label{mainct}
\sum_{\Gamma\in {G}_{2,n}^{ct}}  \frac{1}{|{\rm Aut}(\Gamma)|} {\xi_{\Gamma}}_* 
\left(\sum_{j=0}^n j! 
\left[
\prod_{i=1}^n \left(1+3\omega_i\right) 
\prod_{v\in V(\Gamma)} e^{-t\lambda} 
\right.\right.\\
\left.\left.
\prod_{e=(h,h')\in E_{1}(\Gamma)} \frac{1-e^{t(\psi_{h}+\psi_{h'})}}{\psi_{h}+\psi_{h'}} 
\prod_{e=(h,h')\in E_2(\Gamma)} \frac{1}{\psi_h-(1+3\omega_{h'})} 
\right]_{t^{j}}
\right),
\end{multline}
where:
\begin{itemize}
	\item $E_1(\Gamma)$ denotes the set of edges of $\Gamma$ that separate $\Gamma$ in two connected subgraphs of genus one; 
	\item $E_2(\Gamma)$ denotes the set of edges of $\Gamma$ that separate $\Gamma$ in a subgraph of genus two and a rational subgraph;
	\item $(h,h')$ denotes the pair of half-edges that form a given edge $e$. When $e\in E_2(\Gamma)$, we define $h$ to be the half-edge that points {\it towards} the subgraph of genus two. \footnote{The edge $e$ is incident to a rational vertex $v$ corresponding to a rational tail: if we think of such rational vertex as part of a rooted tree, then $h$ is the half-edge $h(v)$ from \S \ref{sec:comb}.}
\end{itemize}

A few remarks before the proof. 

\begin{rem} \label{rem:geom}
Inside the push-forward ${\xi_{\Gamma}}_*$ in the above formula, the factor 
\[
 \frac{1}{\psi_h-(1+3\omega_{h'})} 
\]
relative to $(h,h')\in E_2(\Gamma)$ is interpreted formally as 
\begin{equation} \label{expand}
 \frac{1}{\psi_h-(1+3\omega_{h'})} = -\frac{1}{1+3\omega_{h'}} \cdot \frac{1}{1-\dfrac{\psi_h}{1+3\omega_{h'}}} 
=  -\frac{1}{1+3\omega_{h'}} \cdot \sum_{k=0}^\infty \left(\frac{\psi_h}{1+3\omega_{h'}} \right)^k.
\end{equation}
This expression should still raise eyebrows because of the class $\omega_{h'}$ at the denominator. But consider a rational tail $T$ of $\Gamma$ as a rooted tree attached to the core $C(\Gamma)$ via an edge $e =(r,r')$, with $r$ incident to $T$. From \S \ref{subsec:dec}, for all edges (resp.~legs) of $T$, $\omega_{h'}$ (resp.~$\omega_i$) coincides with  $\omega_{r'}$. Hence the expression 
\[
\prod_{i=1}^n \left(1+3\omega_i\right) \prod_{(h,h')\in E_2(T)} \frac{1}{\psi_h-(1+3\omega_{h'})} 
\]
as expanded in \eqref{expand} is  a polynomial in $\omega_{r'}$ and the classes $\psi_h$.
\end{rem}

\begin{rem}
Formula \eqref{mainct} is an extension of the formula in Proposition \ref{prod}. For a graph $\Gamma\in G_{2,n}^{ct}$ with no rational tails, it follows from the recursive description of the class  $\left[\mathcal{H}yp_{2,n}^{ct}\right]$ in \eqref{receq} that the contribution of $\Gamma$ to (\ref{mainct}) is equal to the contribution of $\Gamma$ to the product $\rho_1^*(\left[\overline{\mathcal{H}yp}_{2,1} \right])  \dotsm \rho_n^*(\left[\overline{\mathcal{H}yp}_{2,1} \right])$, which is computed in Proposition \ref{prod}. 
\end{rem}

\begin{rem}
\label{rewB}
An alternative formula for $\mathcal{H}yp_{2,n}^{ct}$ indexed by graphs of rational tails is the degree $n$ part of
\[
\sum_{\Gamma\in {G}^{rt}_{2,n}}  \frac{1}{|{\rm Aut}(\Gamma)|} 
\sum_{j=0}^{n-|E_2(\Gamma)|} 
{\xi_{\Gamma}}_*
\left(
\prod_{i=1}^n \left(1+3\omega_i\right) 
\prod_{(h,h')\in E_2(\Gamma)} \frac{1}{\psi_h-(1+3\omega_{h'})} 
\right)
\cdot (-\lambda - \delta_1)^{j}.
\]
This follows from formula (\ref{pixton}) applied to compute the powers $(-\lambda - \delta_1)^{j}$. 
Here the class $\delta_1$ denotes the pull-back of the class $\delta_1$ via the forgetful map $\overline{\mathcal{M}}_{2,n}\rightarrow \overline{\mathcal{M}}_{2}$.
\end{rem}

\begin{proof}[Proof of Theorem \ref{mainthmctintro}]
When $n=1$, the formula recovers the class 
\[
\left[\mathcal{H}yp^{ct}_{2,1} \right]=  3\omega-\lambda-\delta_1
\]
 computed in \cite[Theorem 2.2]{MR910206} --- in fact, the right-hand side represents also the  class $\left[\overline{{\mathcal{H}yp}}_{2,1} \right]\in A^1(\overline{{\mathcal{M}}}_{2,1})$. Suppose then the statement is true for $n-1$. In order to compute the class of $\mathcal{H}yp^{ct}_{2,n}$, we restrict to compact type the recursion in  Proposition \ref{rec}, to obtain:
\begin{equation}
\label{receq}
\left[\mathcal{H}yp^{ct}_{2,n} \right]
=
\pi_n^*\left[\mathcal{H}yp^{ct}_{2,n-1}\right]  \cdot  \rho_n^*\left[\mathcal{H}yp^{ct}_{2,1}\right] 
- \sum_{i=1}^{n-1} {\sigma_{i}}_\ast\left[ \mathcal{H}yp^{ct}_{2,n-1} \right].
\end{equation}

We now analyze three separate cases.
In {\it Step 1} we check the formula for graphs where the $n$-th leg is not attached to a rational tail. In {\it Step 2} we prove the formula holds for graphs where the $n$-th leg is attached to an external, trivalent, rational vertex. In {\it Step 3} we check the remaining cases, i.e.~when the $n$-th leg is attached to an external, rational subtree in   
$G_{0,m+1}^{ne}$, for some $m\leq n$ (recall the definition of $G_{0,m+1}^{ne}$ from \S \ref{sec:comb}).

\vskip4pt

\noindent \textbf{Step \refstepcounter{case}\arabic{case}.} 
Define the partial order $\leq_1$ on $G_{2,n}^{ct}$ as follows: $\Gamma'\leq_1 \Gamma$ if $\Gamma$ is the graph obtained from $\Gamma'$ by contracting some edges in $E_1(\Gamma')$. Observe that graphs of rational tails type are maximal in this partial order, and that for every $\Gamma' \in G_{2,n}^{ct}$, there is a unique $\Gamma \in G_{2,n}^{rt}$ with $\Gamma'\leq_1 \Gamma$, obtained by contracting all edges in $E_1(\Gamma')$. 

Fix $\Gamma\in G_{2,n}^{rt}$ such that the leg $n$ is not attached to  a rational tail.
Therefore, the leg $n$ is not incident to a vertex part of a rational tail  for all $\Gamma'\in G_{2,n}^{ct}$ with $\Gamma'\leq_1 \Gamma$.
The contribution to the right-hand side of (\ref{receq})  of any $\Gamma'\in G_{2,n}^{ct}$ with $\Gamma'\leq_1 \Gamma$ coincides with the contribution of $\Gamma'$ to 
\[
\pi_n^*(\left[\mathcal{H}yp^{ct}_{2,n-1} \right])\cdot \rho_n^*(\left[\mathcal{H}yp^{ct}_{2,1} \right]).
\]
By Remark \ref{rewB}, such contribution equals the contribution of $\Gamma'$ to
\[
\frac{1}{|{\rm Aut}(\Gamma)|} 
\sum_{j=0}^{n-1}
{\xi_{\Gamma}}_*
\left(
\prod_{i=1}^{n-1} \left(1+3\omega_i\right) 
\prod_{(h,h')\in E_2(\Gamma)} \frac{1}{\psi_h-(1+3\omega_{h'})} 
\right)
\cdot (-\lambda - \delta_1)^{j} \cdot (3\omega_n -\lambda - \delta_1).
\]
We are using that $\pi_n^*({\xi_{\Gamma}}_*(\psi_h))={\xi_{\Gamma}}_*(\psi_h)$ for all $(h,h')\in E_2(\Gamma)$, from the assumption on the leg $n$.
The degree $n$ part of the above formula  coincides with the degree $n$ part of the following 
\[
\frac{1}{|{\rm Aut}(\Gamma)|} 
\sum_{j=0}^{n}
{\xi_{\Gamma}}_*
\left(
\prod_{i=1}^{n} \left(1+3\omega_i\right) 
\prod_{(h,h')\in E_2(\Gamma)} \frac{1}{\psi_h-(1+3\omega_{h'})} 
\right)
\cdot (-\lambda - \delta_1)^{j}.
\]
Hence, again by Remark \ref{rewB}, the contribution of $\Gamma'$ is verified, for all $\Gamma'\leq_1 \Gamma$.

\vskip4pt

\noindent \textbf{Step \refstepcounter{case}\arabic{case}\label{Delta02}.} Fix $\Gamma\in \widetilde{G}_{2,n}$ such that the legs $n$ and $j$ are the only two  legs incident to an external vertex $v_0$ of a rational tail, for some $j\in\{1,\dots,n-1\}$. 
That is, ${\rm Im}(\xi_\Gamma) \subseteq \Delta_{0,\{j,n\}}$, for some $j\in\{1,\dots,n-1\}$.
Let $e_0=(h_0,h'_0)$ be the unique edge incident to $v_0$. The contribution of $\Gamma$ to the right-hand side of (\ref{receq}) is equal to the contribution of $\Gamma$ to
\[
- \sigma_{j *}\left( \mathcal{H}yp^{ct}_{2,n-1} \right)
\]
and by the recursive assumption this equals the degree $n$ part of
\begin{multline}
\label{Gamma02}
-\frac{1}{|{\rm Aut}(\Gamma)|} {\xi_{\Gamma}}_*
\left(\sum_{j=0}^n j! 
\left[
\prod_{i=1}^{n-1} \left(1+3\omega_i\right) 
\prod_{v\in V(\Gamma)} e^{-t\lambda} 
\right.\right.\\
\left.\left.
\prod_{e=(h,h')\in E_{1}(\Gamma)} \frac{1-e^{t(\psi_{h}+\psi_{h'})}}{\psi_{h}+\psi_{h'}} 
\prod_{e\not=e_0\in E_2(\Gamma)} \frac{1}{\psi_h-(1+3\omega_{h'})} 
\right]_{t^{j}}
\right).
\end{multline}
Since $\psi_{h_0}=0$, we have
\[
(1+3\omega_n) \cdot \frac{1}{\psi_{h_0}-(1+3\omega_{h'_0})} =-1.
\]
Hence, (\ref{Gamma02}) and the contribution of $\Gamma$ to (\ref{mainct}) coincide.

\vskip4pt

\noindent \textbf{Step \refstepcounter{case}\arabic{case}\label{cancel}.}  
We analyze the contributions of all graphs with the leg $n$ incident to a rational tail $T$. 
After \textit{Step \ref{Delta02}}, we may assume that $T\in G_{0,m+1}^{ne}$, for some $m \leq n$. The contribution of any such graph to the right-hand side of (\ref{receq}) is equal to its contribution to
\[
\pi_n^*(\left[\mathcal{H}yp^{ct}_{2,n-1} \right])\cdot \rho_n^*(\left[\mathcal{H}yp^{ct}_{2,1} \right]).
\]
We therefore fix $\Gamma$  a graph of rational tail type, and use the expression from Remark \ref{rewB} to prove that the contributions are correct for all $\Gamma' \leq_1 \Gamma$.

If $T$ denotes  the maximal rational subtree of $\Gamma$ to which the leg $n$ is attached, denote  by $G$ the graph obtained by removing $T$ from $\Gamma$.  We do this by  cutting the edge connecting $T$ to $G$ into two half-edges $r$ and $r'$, with $r$ incident to $T$, and $r'$ incident to a vertex of genus $2$ in $G$.  Denote by $\mathfrak{G}$ the set of all graphs $\Gamma = T \bigcup_{e = (r,r')} G$, where $G$ is fixed and $T$ varies among all graphs in $G_{0,m+1}^{ne}$. From Proposition~\ref{eulergen} and Remark \ref{rem:geom}, we have
\begin{align}
\label{symplif}
&\sum_{T\in G_{0,m+1}^{ne}}  \frac{1}{|{\rm Aut}(\Gamma)|} 
{\xi_{\Gamma}}_*
\left(
\prod_{i=1}^n \left(1+3\omega_i\right) 
\prod_{(h,h')\in E_2(\Gamma)} \frac{1}{\psi_h-(1+3\omega_{h'})} 
\right) 
\nonumber \\
&=
\sum_{T\in G_{0,m+1}^{ne}}  \frac{1}{|{\rm Aut}(\Gamma)|} 
{\xi_{\Gamma}}_*
\left(
\prod_{i=1}^n \left(1+3\omega_i\right) 
\prod_{v\in V(T)} \frac{1}{\psi_{h(v)}-(1+3\omega_{r'})}  
\prod_{(h,h')\in E_2(\Gamma)} \frac{1}{\psi_h-(1+3\omega_{h'})} 
\right) 
\nonumber \\
&=
\sum_{T\in B_{0,m+1}}  \frac{(-1)^{|V(T)|}}{|{\rm Aut}(\Gamma)|} 
{\xi_{\Gamma}}_*
\left(
\prod_{i=1}^n \left(1+3\omega_i\right)  
\prod_{v\in V(T)} \frac{1}{1+3\omega_{r'}}  
\prod_{(h,h')\in E_2(G)} \frac{1}{\psi_h-(1+3\omega_{h'})} 
\right).
\end{align}
 From (\ref{symplif}) and Remark \ref{rewB}, 
the contribution of a graph $\Gamma$  in $\mathfrak{G}$ to (\ref{mainct}) is 
\begin{align}
\label{cont1}
\frac{(-1)^{|V(T)|}}{|{\rm Aut}(\Gamma)|} 
\sum_{j=0}^{n}
{\xi_{\Gamma}}_*
\left(
\prod_{i=1}^{n} \left(1+3\omega_i\right) 
\prod_{v\in V(T)} \frac{1}{1+3\omega_{r'}}
\prod_{(h,h')\in E_2(G)} \frac{1}{\psi_h-(1+3\omega_{h'})} 
\right)
\cdot (-\lambda - \delta_1)^{j}.
\end{align}
On the other hand, using (\ref{symplif}) and Remark \ref{rewB} for the class of $\mathcal{H}yp^{ct}_{2,n-1}$, the contribution of a graph $\Gamma$ in $\mathfrak{G}$ to $
\pi_n^*(\left[\mathcal{H}yp^{ct}_{2,n-1} \right])\cdot \rho_n^*(\left[\mathcal{H}yp^{ct}_{2,1} \right])$ 
 equals 
\begin{multline}
\label{cont2}
\frac{(-1)^{|V(T)|}}{|{\rm Aut}(\Gamma)|} 
\sum_{j=0}^{n-1}
{\xi_{\Gamma}}_*
\left(
\prod_{i=1}^{n-1} \left(1+3\omega_i\right) 
\prod_{v\in V(T)} \frac{1}{1+3\omega_{r'}}
\prod_{(h,h')\in E_2(G)} \frac{1}{\psi_h-(1+3\omega_{h'})} 
\right) \\
\cdot (-\lambda - \delta_1)^{j} \cdot (3\omega_n -\lambda - \delta_1).
\end{multline}
Note that $\pi_n^*(\psi_h)=\psi_h$ for all $(h,h')\in E_2(G)$, since the leg $n$ is incident to $T$. The  degree $n$ part of (\ref{cont1}) and (\ref{cont2}) coincide.  This concludes the proof of Theorem \ref{mainthmctintro}.
\end{proof}

\section{The closure beyond compact type}
\label{sec:NCTn}

In this section we complete the computation of the classes of the loci $\overline{\mathcal{H}yp}_{2,n}$ by describing the contributions of decorated boundary strata classes of non-compact type. As in the study of the contributions of decorated boundary strata classes of compact type in \S \ref{sec:ct}, we use the recursive description from Proposition \ref{rec}.

\subsection{The classes \texorpdfstring{$\Phi_i$}{Phi} and \texorpdfstring{$\Gamma_{i,j}$}{Gamma}}
Recall the classes $\Phi_i$ and $\Gamma_{i,j}$ defined in \S \ref{sec:rec}.
Define
\begin{eqnarray}
\Phi\Gamma_n & := & \sum_{i} \Phi_i 
+
\sum_{i<j} \Gamma_{i,j}.
\end{eqnarray}
Our first aim is to find an explicit expression for the class $\Phi\Gamma_n$.

Let ${\mathcal{H}}_{1,\pm\cup I}$ be the locus of elliptic curves with marked points $p_+, p_-$ and $p_i$ for $i\in I$ such that $2p_i\sim p_+ + p_-$, for all $i\in I$. For $I\subset [n]$, let
\[
\varphi \colon \overline{\mathcal{M}}_{1,\pm\cup I} \times \overline{\mathcal{M}}_{0,\pm\cup I^c} \rightarrow \overline{\mathcal{M}}_{2,[n]}
\]
be the glueing map identifying the points $p_+$ together, and the points $p_-$ together.
Given $A\in A^*(\overline{\mathcal{M}}_{1,\pm\cup I})$ and $B\in A^*(\overline{\mathcal{M}}_{0,\pm\cup I^c})$,
let 
\[
A \boxtimes B:= \pi_1^*(A) \cdot \pi_2^*(B) \in A^*(\overline{\mathcal{M}}_{1,\pm\cup I} \times \overline{\mathcal{M}}_{0,\pm\cup I^c}),
\]
where $\pi_1\colon \overline{\mathcal{M}}_{1,\pm\cup I} \times \overline{\mathcal{M}}_{0,\pm\cup I^c} \rightarrow \overline{\mathcal{M}}_{1,\pm\cup I}$ and $\pi_2 \colon \overline{\mathcal{M}}_{1,\pm\cup I} \times \overline{\mathcal{M}}_{0,\pm\cup I^c} \rightarrow \overline{\mathcal{M}}_{0,\pm\cup I^c}$ are the two projections.
By definition, the first description of the classes $\Phi_i$ and $\Gamma_{i,j}$ is the following one. 

\begin{lem}
\label{lem:phigamma}
One has
\begin{eqnarray*}
\Phi_i & = & \frac{1}{2} \varphi_* \left(\left[\overline{\mathcal{H}}_{1,\pm\cup \{i,n\}^c}\right] \boxtimes
\mbox{
\begin{tikzpicture}[baseline={([yshift=-.7ex]current bounding box.center)}]
      \path(0,0) ellipse (1 and 1);     
      \tikzstyle{level 1}=[clockwise from=45,level distance=8mm,sibling angle=90]
      \node [draw,circle,fill] (A0) at (0:0) {}
      child {node [label=45: {$\scriptstyle{i}$}]{}}
      child {node [label=-45: {$\scriptstyle{n}$}]{}}
      child {node [label=-135: {$\scriptstyle{-}$}]{}}
      child {node [label=135: {$\scriptstyle{+}$}]{}};            
    \end{tikzpicture}
}
\right),\\
\Gamma_{i,j} & = & \frac{1}{2} \varphi_* \left(\left[\overline{\mathcal{H}}_{1,\pm\cup \{i,j,n\}^c}\right] \boxtimes
\left(
\mbox{
\begin{tikzpicture}[baseline={([yshift=-.7ex]current bounding box.center)}]
      \path(0,0) ellipse (1 and 1);     
      \tikzstyle{level 1}=[clockwise from=180,level distance=8mm,sibling angle=120]
      \node [draw,circle,fill] (A0) at (90:1) {}
      child {node [label=180: {$\scriptstyle{+}$}]{}}
      child {node [label=60: {$\scriptstyle{j}$}]{}}; 
      \tikzstyle{level 1}=[clockwise from=0,level distance=8mm,sibling angle=90]
      \node [draw,circle,fill] (A1) at (-90:1) {}
      child {node [label=0: {$\scriptstyle{i}$}]{}}
      child {node [label=-90: {$\scriptstyle{n}$}]{}}
      child {node [label=180: {$\scriptstyle{-}$}]{}};            
      \path (A0) edge [] (A1);
    \end{tikzpicture}
}
+
\mbox{
\begin{tikzpicture}[baseline={([yshift=-.7ex]current bounding box.center)}]
      \path(0,0) ellipse (1 and 1);     
      \tikzstyle{level 1}=[clockwise from=180,level distance=8mm,sibling angle=120]
      \node [draw,circle,fill] (A0) at (90:1) {}
      child {node [label=180: {$\scriptstyle{+}$}]{}}
      child {node [label=60: {$\scriptstyle{i}$}]{}}; 
      \tikzstyle{level 1}=[clockwise from=0,level distance=8mm,sibling angle=90]
      \node [draw,circle,fill] (A1) at (-90:1) {}
      child {node [label=0: {$\scriptstyle{j}$}]{}}
      child {node [label=-90: {$\scriptstyle{n}$}]{}}
      child {node [label=180: {$\scriptstyle{-}$}]{}};            
      \path (A0) edge [] (A1);
    \end{tikzpicture}
}
\right)
\right).
\end{eqnarray*}
\end{lem}

For $I\subset [n]$, let $\left[\overline{\mathcal{H}yp}_{2,I}\right]$ be the pull-back of the class of the closure of the locus of genus two curves with $|I|$ marked Weierstrass points
via the forgetful map $\overline{\mathcal{M}}_{2,[n]}\rightarrow \overline{\mathcal{M}}_{2,I}$.

The following Theorem expresses the classes $\Phi_i$ and $\Gamma_{i,j}$ as a linear combination of classes of type $\left[\overline{\mathcal{H}yp}_{2,I}\right]$ with boundary strata classes as coefficients.

\begin{thm}
\label{thm:phigamma}
For $n\leq 6$, one has
\begin{eqnarray*}
\Phi\Gamma_n & = & \sum_{i\in [n-1]} \left[\overline{\mathcal{H}yp}_{2,\{i,n\}^c}\right] \cdot
\mbox{
\begin{tikzpicture}[baseline={([yshift=-.7ex]current bounding box.center)}]
      \path(0,0) circle (1 and 1);
      \tikzstyle{level 1}=[counterclockwise from=-60,level distance=8mm,sibling angle=120]
      \node[draw,circle, fill] (A0) at (0:1.5) {} 
      child {node [label=-60: {$\scriptstyle{n}$}]{}}
      child {node [label=60: {$\scriptstyle{i}$}]{}};
      \tikzstyle{level 1}=[counterclockwise from=120,level distance=10mm,sibling angle=40]
      \node[draw,circle,inner sep=2.5] (A1) at (180:1.5) {$\scriptstyle{1}$}
      child {node [label=0: {}]{}}
      child {node [label=0: {}]{}}
      child {node [label=0: {}]{}}      
      child {node [label=0: {}]{}};
      \path (A0) edge [bend left=-45.000000] node[auto,near start,above=1]{}(A1);
      \path (A0) edge [bend left=45.000000] node[auto,near start,below=1]{}(A1);
         \end{tikzpicture}
}
-\frac{1}{2} \sum_{i,j \in [n-1]} \left[\overline{\mathcal{H}yp}_{2,\{i,j,n\}^c}\right] \cdot
\mbox{
\begin{tikzpicture}[baseline={([yshift=-.7ex]current bounding box.center)}]
      \path(0,0) ellipse (2 and 2);
      \tikzstyle{level 1}=[counterclockwise from=120,level distance=10mm,sibling angle=60]
      \node [draw,circle,inner sep=2.5] (A0) at (180:1.5) {$\scriptstyle{1}$} 
      child {node [label=0: {}]{}}
      child {node [label=0: {}]{}}
      child {node [label=0: {}]{}};
      \tikzstyle{level 1}=[counterclockwise from=60,level distance=8mm,sibling angle=60]
      \node [draw,circle,fill] (A1) at (60:1.5) {}
      child {node [label=60: {$\scriptstyle{j}$}]{}};
      \tikzstyle{level 1}=[clockwise from=-30,level distance=8mm,sibling angle=60]
      \node [draw,circle,fill] (A2) at (-60:1.5) {}
      child {node [label=0: {$\scriptstyle{i}$}]{}}
      child {node [label=-90: {$\scriptstyle{n}$}]{}};      
      \path (A0) edge [] (A1);
      \path (A0) edge [] (A2);
      \path (A1) edge [] (A2);      
    \end{tikzpicture}
}
\\
&& +\frac{1}{2} \sum_{i,j,k \in [n-1]} \left[\overline{\mathcal{H}yp}_{2,\{i,j,k,n\}^c}\right] \cdot
\mbox{
\begin{tikzpicture}[baseline={([yshift=-.7ex]current bounding box.center)}]
      \path(0,0) ellipse (2 and 2);
      \tikzstyle{level 1}=[counterclockwise from=150,level distance=10mm,sibling angle=60]
      \node [draw,circle,inner sep=2.5] (A0) at (180:1.5) {$\scriptstyle{1}$} 
      child {node [label=0: {}]{}}
      child {node [label=0: {}]{}};
      \tikzstyle{level 1}=[counterclockwise from=90,level distance=8mm,sibling angle=60]
      \node [draw,circle,fill] (A1) at (90:1.5) {}
      child {node [label=90: {$\scriptstyle{j}$}]{}};
      \tikzstyle{level 1}=[counterclockwise from=-45,level distance=8mm,sibling angle=90]
      \node [draw,circle,fill] (A2) at (0:1.5) {}
      child {node [label=-45: {$\scriptstyle{n}$}]{}}
      child {node [label=45: {$\scriptstyle{i}$}]{}};
      \tikzstyle{level 1}=[clockwise from=-90,level distance=8mm,sibling angle=60]
      \node [draw,circle,fill] (A3) at (-90:1.5) {}
      child {node [label=-90: {$\scriptstyle{k}$}]{}};      
      \path (A0) edge [] (A1);
      \path (A1) edge [] (A2);
      \path (A2) edge [] (A3);      
      \path (A3) edge [] (A0);            
    \end{tikzpicture}
}
\\
&& +\frac{1}{4} \sum_{i,j,k,l \in [n-1]} \left[\overline{\mathcal{H}yp}_{2,\{i,j,k,l,n\}^c}\right] \cdot
\left(
\mbox{
\begin{tikzpicture}[baseline={([yshift=-.7ex]current bounding box.center)}]
      \path(0,0) ellipse (2 and 2);
      \tikzstyle{level 1}=[counterclockwise from=180,level distance=10mm,sibling angle=60]
      \node [draw,circle,inner sep=2.5] (A0) at (180:1.5) {$\scriptstyle{1}$} 
      child {node [label=0: {}]{}};
      \tikzstyle{level 1}=[counterclockwise from=108,level distance=8mm,sibling angle=60]
      \node [draw,circle,fill] (A1) at (108:1.5) {}
      child {node [label=108: {$\scriptstyle{l}$}]{}};
      \tikzstyle{level 1}=[counterclockwise from=36,level distance=8mm,sibling angle=60]
      \node [draw,circle,fill] (A2) at (36:1.5) {}
      child {node [label=36: {$\scriptstyle{k}$}]{}};
      \tikzstyle{level 1}=[counterclockwise from=-36,level distance=8mm,sibling angle=60]
      \node [draw,circle,fill] (A3) at (-36:1.5) {}
      child {node [label=-36: {$\scriptstyle{j}$}]{}};
      \tikzstyle{level 1}=[clockwise from=-78,level distance=8mm,sibling angle=60]
      \node [draw,circle,fill] (A4) at (-108:1.5) {}
      child {node [label=-78: {$\scriptstyle{i}$}]{}}
      child {node [label=-138: {$\scriptstyle{n}$}]{}};      
      \path (A0) edge [] (A1);
      \path (A1) edge [] (A2);
      \path (A2) edge [] (A3);      
      \path (A3) edge [] (A4);            
      \path (A0) edge [] (A4);            
    \end{tikzpicture}
}
-3
\mbox{
\begin{tikzpicture}[baseline={([yshift=-.7ex]current bounding box.center)}]
      \path(0,0) ellipse (2 and 2);
      \tikzstyle{level 1}=[counterclockwise from=180,level distance=10mm,sibling angle=60]
      \node [draw,circle,inner sep=2.5] (A0) at (180:1.5) {$\scriptstyle{1}$} 
      child {node [label=0: {}]{}};
      \tikzstyle{level 1}=[counterclockwise from=108,level distance=8mm,sibling angle=60]
      \node [draw,circle,fill] (A1) at (108:1.5) {}
      child {node [label=108: {$\scriptstyle{k}$}]{}};
      \tikzstyle{level 1}=[counterclockwise from=36,level distance=8mm,sibling angle=60]
      \node [draw,circle,fill] (A2) at (36:1.5) {}
      child {node [label=36: {$\scriptstyle{j}$}]{}};
      \tikzstyle{level 1}=[counterclockwise from=-66,level distance=8mm,sibling angle=60]
      \node [draw,circle,fill] (A3) at (-36:1.5) {}
      child {node [label=-66: {$\scriptstyle{n}$}]{}}
      child {node [label=0: {$\scriptstyle{i}$}]{}};
      \tikzstyle{level 1}=[clockwise from=-108,level distance=8mm,sibling angle=60]
      \node [draw,circle,fill] (A4) at (-108:1.5) {}
      child {node [label=-108: {$\scriptstyle{l}$}]{}};      
      \path (A0) edge [] (A1);
      \path (A1) edge [] (A2);
      \path (A2) edge [] (A3);      
      \path (A3) edge [] (A4);            
      \path (A0) edge [] (A4);            
    \end{tikzpicture}
}
\right)
\\
&& +\frac{1}{4} \sum_{i,j,k,l,m \in [n-1]}
\left(
3
\mbox{
\begin{tikzpicture}[baseline={([yshift=-.7ex]current bounding box.center)}]
      \path(0,0) ellipse (2 and 2);
      \tikzstyle{level 1}=[counterclockwise from=180,level distance=10mm,sibling angle=60]
      \node [draw,circle,inner sep=2.5] (A0) at (180:1.5) {$\scriptstyle{1}$};
      \tikzstyle{level 1}=[counterclockwise from=120,level distance=8mm,sibling angle=60]
      \node [draw,circle,fill] (A1) at (120:1.5) {}
      child {node [label=120: {$\scriptstyle{k}$}]{}};
      \tikzstyle{level 1}=[counterclockwise from=60,level distance=8mm,sibling angle=60]
      \node [draw,circle,fill] (A2) at (60:1.5) {}
      child {node [label=60: {$\scriptstyle{j}$}]{}};
      \tikzstyle{level 1}=[counterclockwise from=-30,level distance=8mm,sibling angle=60]
      \node [draw,circle,fill] (A3) at (0:1.5) {}
      child {node [label=-30: {$\scriptstyle{n}$}]{}}
      child {node [label=30: {$\scriptstyle{i}$}]{}};
      \tikzstyle{level 1}=[clockwise from=-60,level distance=8mm,sibling angle=60]
      \node [draw,circle,fill] (A4) at (-60:1.5) {}
      child {node [label=-60: {$\scriptstyle{m}$}]{}};      
      \tikzstyle{level 1}=[clockwise from=-120,level distance=8mm,sibling angle=60]
      \node [draw,circle,fill] (A5) at (-120:1.5) {}
      child {node [label=-120: {$\scriptstyle{l}$}]{}};            
      \path (A0) edge [] (A1);
      \path (A1) edge [] (A2);
      \path (A2) edge [] (A3);      
      \path (A3) edge [] (A4);            
      \path (A4) edge [] (A5);            
      \path (A5) edge [] (A0);            
    \end{tikzpicture}
}
-
\mbox{
\begin{tikzpicture}[baseline={([yshift=-.7ex]current bounding box.center)}]
      \path(0,0) ellipse (2 and 2);
      \tikzstyle{level 1}=[counterclockwise from=180,level distance=10mm,sibling angle=60]
      \node [draw,circle,inner sep=2.5] (A0) at (180:1.5) {$\scriptstyle{1}$};
      \tikzstyle{level 1}=[counterclockwise from=120,level distance=8mm,sibling angle=60]
      \node [draw,circle,fill] (A1) at (120:1.5) {}
      child {node [label=120: {$\scriptstyle{m}$}]{}};
      \tikzstyle{level 1}=[counterclockwise from=60,level distance=8mm,sibling angle=60]
      \node [draw,circle,fill] (A2) at (60:1.5) {}
      child {node [label=60: {$\scriptstyle{l}$}]{}};
      \tikzstyle{level 1}=[counterclockwise from=0,level distance=8mm,sibling angle=60]
      \node [draw,circle,fill] (A3) at (0:1.5) {}
      child {node [label=0: {$\scriptstyle{k}$}]{}};
      \tikzstyle{level 1}=[clockwise from=-60,level distance=8mm,sibling angle=60]
      \node [draw,circle,fill] (A4) at (-60:1.5) {}
      child {node [label=-60: {$\scriptstyle{j}$}]{}};      
      \tikzstyle{level 1}=[clockwise from=-90,level distance=8mm,sibling angle=60]
      \node [draw,circle,fill] (A5) at (-120:1.5) {}
      child {node [label=-90: {$\scriptstyle{i}$}]{}}
      child {node [label=-150: {$\scriptstyle{n}$}]{}};            
      \path (A0) edge [] (A1);
      \path (A1) edge [] (A2);
      \path (A2) edge [] (A3);      
      \path (A3) edge [] (A4);            
      \path (A4) edge [] (A5);            
      \path (A5) edge [] (A0);            
    \end{tikzpicture}
}
\right).
\end{eqnarray*}
\end{thm}

In the above formula, we use the following convention on graphs. 
For simplicity, all graphs in the statement are drawn with six legs. 
All unmarked legs are attached to the elliptic vertex. 
If $n<6$, one drops the appropriate number of legs attached to the elliptic vertex, {and truncates the formula ignoring the part where the rational components all together contain more than $n$ legs.}
For each graph, one distributes the remaining markings to the legs attached to the elliptic vertex
(up to isomorphism, there is a unique way to do this).

\begin{proof}
The following identity
\begin{multline}
\label{admcoverint}
\left[\overline{\mathcal{H}yp}_{2,\{i,n\}^c}\right] \cdot
\mbox{
\begin{tikzpicture}[baseline={([yshift=-.7ex]current bounding box.center)}]
      \path(0,0) circle (1 and 1);
      \tikzstyle{level 1}=[counterclockwise from=-60,level distance=8mm,sibling angle=120]
      \node[draw,circle, fill] (A0) at (0:1.5) {} 
      child {node [label=-60: {$\scriptstyle{n}$}]{}}
      child {node [label=60: {$\scriptstyle{i}$}]{}};
      \tikzstyle{level 1}=[counterclockwise from=120,level distance=10mm,sibling angle=40]
      \node[draw,circle,inner sep=2.5] (A1) at (180:1.5) {$\scriptstyle{1}$}
      child {node [label=0: {}]{}}
      child {node [label=0: {}]{}}
      child {node [label=0: {}]{}}      
      child {node [label=0: {}]{}};
      \path (A0) edge [bend left=-45.000000] node[auto,near start,above=1]{}(A1);
      \path (A0) edge [bend left=45.000000] node[auto,near start,below=1]{}(A1);
         \end{tikzpicture}
}
 =  
\frac{1}{2} \varphi_* \left(\left[\overline{\mathcal{H}}_{1,\pm\cup \{i,n\}^c}\right] \boxtimes
\mbox{
\begin{tikzpicture}[baseline={([yshift=-.7ex]current bounding box.center)}]
      \path(0,0) ellipse (1 and 1);     
      \tikzstyle{level 1}=[clockwise from=45,level distance=8mm,sibling angle=90]
      \node [draw,circle,fill] (A0) at (0:0) {}
      child {node [label=45: {$\scriptstyle{i}$}]{}}
      child {node [label=-45: {$\scriptstyle{n}$}]{}}
      child {node [label=-135: {$\scriptstyle{-}$}]{}}
      child {node [label=135: {$\scriptstyle{+}$}]{}};            
    \end{tikzpicture}
}
\right)
\\
+ \sum_{j} 
\varphi_* \left(\left[\overline{\mathcal{H}}_{1,\pm\cup \{i,j,n\}^c}\right] \boxtimes
\mbox{
\begin{tikzpicture}[baseline={([yshift=-.7ex]current bounding box.center)}]
      \path(0,0) ellipse (1 and 1);     
      \tikzstyle{level 1}=[clockwise from=180,level distance=8mm,sibling angle=120]
      \node [draw,circle,fill] (A0) at (90:1) {}
      child {node [label=180: {$\scriptstyle{+}$}]{}}
      child {node [label=60: {$\scriptstyle{j}$}]{}}; 
      \tikzstyle{level 1}=[clockwise from=0,level distance=8mm,sibling angle=90]
      \node [draw,circle,fill] (A1) at (-90:1) {}
      child {node [label=0: {$\scriptstyle{i}$}]{}}
      child {node [label=-90: {$\scriptstyle{n}$}]{}}
      child {node [label=180: {$\scriptstyle{-}$}]{}};            
      \path (A0) edge [] (A1);
    \end{tikzpicture}
}
\right)
+\sum_{j<k} 
\varphi_* \left(\left[\overline{\mathcal{H}}_{1,\pm\cup \{i,j,k,n\}^c}\right] \boxtimes
\mbox{
\begin{tikzpicture}[baseline={([yshift=-.7ex]current bounding box.center)}]
     \path(0,0) ellipse (1 and 2);     
      \tikzstyle{level 1}=[clockwise from=180,level distance=8mm,sibling angle=120]
      \node [draw,circle,fill] (A0) at (120:1.5) {}
      child {node [label=180: {$\scriptstyle{+}$}]{}}
      child {node [label=60: {$\scriptstyle{k}$}]{}}; 
      \tikzstyle{level 1}=[clockwise from=0,level distance=8mm,sibling angle=90]
      \node [draw,circle,fill] (A1) at (0:0) {}
      child {node [label=0: {$\scriptstyle{j}$}]{}};
      \tikzstyle{level 1}=[clockwise from=-60,level distance=8mm,sibling angle=60]
      \node [draw,circle,fill] (A2) at (-120:1.5) {}
      child {node [label=-60: {$\scriptstyle{i}$}]{}}
      child {node [label=-120: {$\scriptstyle{n}$}]{}}
      child {node [label=180: {$\scriptstyle{-}$}]{}};            
      \path (A0) edge [] (A1);
      \path (A1) edge [] (A2);
    \end{tikzpicture}
}
\right)
\end{multline}
can be verified by using the space of admissible covers to study the intersection on the left-hand side.
Indeed, there are three types of components in the intersection. On the right-hand side, the first summand coincides with the case when all points with labels in $\{i,n\}^c$ lie on a smooth elliptic component. 
The second summand coincides with the case when one of the points with labels in $\{i,n\}^c$ collides towards one of the two singular points on the elliptic component; the resulting stable curves (are stably equivalent to curves which) admit an admissible double cover ramified at all points with labels in $\{i,n\}^c$. Finally, the third summand arises when two of the points with labels in $\{i,n\}^c$ collide towards one of the two singular points on the elliptic component. 
The intersection is transversal along the first type of component and has multiplicity $2$ along the other components (each singular point gives one contribution). On the other hand, all glueing maps $\varphi$ have degree $2$. This explains the coefficients in the identity.

Using Lemma \ref{lem:phigamma} to express the classes $\Phi_i$ and (\ref{admcoverint}), we have
\begin{eqnarray*}
\Phi\Gamma_n &=&
\sum_{i\in [n-1]} 
\left[\overline{\mathcal{H}yp}_{2,\{i,n\}^c}\right] \cdot
\mbox{

}
\right)
+\sum_{i<j} \Gamma_{i,j}.
\end{eqnarray*}
Using in addition Lemma \ref{lem:phigamma} to express the classes $\Gamma_{i,j}$, we have
\begin{eqnarray}
\Phi\Gamma_n
&=&
\sum_{i\in [n-1]} 
\left[\overline{\mathcal{H}yp}_{2,\{i,n\}^c}\right] \cdot
\mbox{
%
}
\right).
\end{eqnarray}
The first summand on the right-hand side matches the first contribution to the formula for $\Phi\Gamma_n$ in Theorem \ref{thm:phigamma}.
In the following we will thus focus on the remaining summands \eqref{PG3}.  The identity
\begin{align*}
\sum_{i,j} \left[\overline{\mathcal{H}yp}_{2,\{i,j,n\}^c}\right] \cdot
\mbox{
%
}
\right)
\right)
\end{align*}
follows from the study of the intersection on the right-hand side via the space of admissible covers, and the argument to prove it is similar to the one used for (\ref{admcoverint}).
Using the identity, \eqref{PG3} becomes
\begin{align}
&-\frac{1}{2} 
\sum_{i,j} \left[\overline{\mathcal{H}yp}_{2,\{i,j,n\}^c}\right] \cdot
\mbox{
%
}
\right)
\right).
\end{align}
The first summand matches the second contribution in the statement of the theorem, hence
we will continue analyzing the expression \eqref{PG4-1} $+$ \eqref{PG4-2}. 

Using the following identity
\begin{align*}
&\sum_{i,j,k \in [n-1]} \left[\overline{\mathcal{H}yp}_{2,\{i,j,k,n\}^c}\right] \cdot
\mbox{
%
}
.
\end{align}
The first summand coincide with the third contribution in the statement of the theorem. Proceeding in a similar fashion, \eqref{PG5-1} becomes
\begin{align}
&+\frac{1}{4} \sum_{i,j,k,l \in [n-1]} \left[\overline{\mathcal{H}yp}_{2,\{i,j,k,l,n\}^c}\right] \cdot
\left(
\mbox{
%
}
\right).
\]
\end{proof}

\begin{rem}
As a check, the formula for $\Phi\Gamma_n$ in Theorem \ref{thm:phigamma} verifies for $3\leq n\leq 6$ the equality
\[
{\pi_{n-1}}_* \left(\Phi\Gamma_n \right)= (8-n) \Phi\Gamma_{n-1}
\]
satisfied by the classes $\Phi\Gamma_n$ after  the formulas in the proof of Proposition \ref{rec}.
\end{rem}

\subsection{The contributions of non-compact type}
\label{subsec:NCTn}
We are now ready to complete the description of the class $\left[\overline{\mathcal{H}yp}_{2,n}\right]$. 
Define
\[
{\rm NCT}_n := \left[\overline{\mathcal{H}yp}_{2,n}\right] - \left[{\mathcal{H}yp}^{ct}_{2,n}\right].
\]
Restricting Proposition \ref{rec} to $A^*(\overline{\mathcal{M}}_{2,n}\setminus \mathcal{M}^{ct}_{2,n})$,
one has
\begin{align}
\label{eqNCTn}
{\rm NCT}_n = \pi_n^*({\rm NCT}_{n-1}) \cdot (3\omega_n - \lambda - \delta_1) - \sum_{i\in[n-1]} \sigma_{i *} ({\rm NCT}_{n-1})- \Phi\Gamma_n.
\end{align}
This  equation and Theorem \ref{thm:phigamma} allow us to recursively compute ${\rm NCT}_n$. In the following, we simplify the resulting expressions when $n\leq 5$. We use the convention on labelling graphs from \S \ref{nota}.

\begin{rem}
As we have already discussed, for $n=1$ one has
\[
\left[\overline{\mathcal{H}yp}_{2,1}\right] = 3\omega -\lambda - \delta_1.
\]
This formula implies that the pull-back of $\left[\overline{\mathcal{H}yp}_{2,1}\right]$ via a forgetful map $\Mbar_{2,n} \rightarrow \Mbar_{2,1}$  consists of polynomials in $\omega $ and $\lambda$ classes restricted to strata which meet transversally the locus of curves of non-compact type in $\Mbar_{2,n}$. The classes $\omega$ and $\lambda$ restrict in a natural way to any stratum in the moduli space, and the intersection of the supporting strata with strata corresponding to dual graphs where all edges are non-separating is transversal. Hence the formulas we provide do not hide any issue of non-transversal/excess intersection.  In the following, we  express ${\rm NCT}_3, {\rm NCT}_4,$ and ${\rm NCT}_5$ as  linear combinations of intersections of compact-type hyperelliptic classes, polynomials in $(\lambda+\delta)$ and boundary strata classes.
\end{rem}

\begin{rem}
When $n=2$, the class $\Phi\Gamma_2$ is supported on a single boundary stratum, and we have
\begin{align}
\label{NCT2}
{\rm NCT}_2 = - \Phi\Gamma_2 = 
-
\mbox{

}.
\end{align}
The resulting expression for $\left[\overline{\mathcal{H}yp}_{2,2}\right]$ is equivalent to the one first computed in \cite{DR}.
\end{rem}

\begin{cor}
\label{cor:NCT3}
When $n=3$, we have
\begin{eqnarray*}
{\rm NCT}_3 &=& -\sum_{i\in [3]} \left[\overline{\mathcal{H}yp}_{2,\{i\}}\right] 
\cdot
\mbox{
%
}.
\end{eqnarray*}
\end{cor}

\begin{proof}
From the recursion, we have
\begin{eqnarray*}
{\rm NCT}_3 &=& \pi_3^*({\rm NCT}_2) \cdot (3\omega_3 - \lambda - \delta_1) - \sum_{i\in[2]} \sigma_{i *} ({\rm NCT}_2)- \Phi\Gamma_3\\
 &=& -\left[\overline{\mathcal{H}yp}_{2,\{3\}}\right] 
\cdot
\mbox{
%
}.
\]
The   equivalence in $A^3\left(\overline{\mathcal{M}}_{2,3}\right)$
\begin{align}
\label{trianglestoban}
\frac{1}{2}\left(
\mbox{
%
},
\end{align}
is obtained as a pullback via $\pi_3$ of the WDVV relations pushed forward from $\overline{\M}_{0,4}$. Applying \eqref{trianglestoban}
one can write ${\rm NCT}_3$ as the symmetric expression in the statement.
\end{proof}

The following statement  concludes the proof of Theorem \ref{mainthm<=4}.

\begin{cor}
\label{cor:NCT4}
When $n=4$, we have
\begin{eqnarray*}
{\rm NCT}_4 &=& 
- \sum_{i<j\leq 4}\left[{\mathcal{H}yp}^{ct}_{2,\{i,j\}}\right]  \cdot
\mbox{
%
}.
\]
From the pull-back of \eqref{NCT2}, we have
\[
\left[\overline{\mathcal{H}yp}_{2,\{i,4\}^c}\right] = \left[{\mathcal{H}yp}^{ct}_{2,\{i,4\}^c}\right] 
-
\mbox{
%
}.
\end{eqnarray*}
The following relation is analogous to \eqref{trianglestoban}:
\begin{align*}
\frac{1}{2}\left(
\mbox{
%
}.
\end{align*}
We also use the following relations in $A^4\left(\overline{\mathcal{M}}_{2,4}\right)$: 
\begin{align}
\mbox{
%
},
\label{psionfork}
\end{align}
 and obtain the symmetric expression in the statement. Relation \eqref{psionm} is obtained using the fact that $\psi = \lambda$ on $\overline{\mathcal{M}}_{1,1}$, and then expressing $\psi$ on $\overline{\mathcal{M}}_{1,2}$ as the pull-back of $\psi$ on $\overline{\mathcal{M}}_{1,1}$ plus a section. Relation \eqref{psionfork} is obtained by pushing-forward two different boundary expressions for a $\psi$
 class on $\overline{\mathcal{M}}_{0,5}$.
\end{proof}

\begin{proof}[Proof of Theorem \ref{mainthm<=4}]
The graphs appearing in the above expression for ${\rm NCT}_n$ when $n\leq 4$ coincide with the graphs in $\widetilde{G}_{2,n}$. As the decorations match the degree $n$ component of the formula in Theorem \ref{mainthm<=4}, the statement follows.
\end{proof}

Theorem  \ref{thm:phigamma} and \eqref{eqNCTn} determine explicit expressions  for the classes ${\rm NCT}_n$ also when $n=5,6$.
When $n=5$, we simplify the resulting expression by means of tautological relations similar to the ones used for $n\leq 4$. We obtain the following formula.

\begin{cor}
\label{cor:NCT5}
${\rm NCT}_5$ is equal to
\begin{align*}
& - \sum_{i<j\leq 5}\left[{\mathcal{H}yp}^{ct}_{2,\{i,j\}^c}\right]  
\mbox{

}
.
\end{align*}
\end{cor}

We have not fully symmetrized the formula for $\textrm{NCT}_5$ with respect to the five markings. 
However, we know that the formula must represent a symmetric class.
Hence, the intersection of the formula with classes of complementary dimension should not depend on the distribution of the markings. As a check, we have verified that this occurs in a number of cases. For instance, the expression in the last two lines of the formula has zero intersection with both classes
\begin{align*}
\lambda
\mbox{
%
}
.
\end{align*}

\bibliographystyle{alphanum}
\bibliography{Biblio.bib}

\end{document}